\documentclass{amsart}
\usepackage{amsmath,latexsym,amsfonts,amssymb,amsthm,mathrsfs,longtable,lscape,dsfont,fancybox,yfonts}
\usepackage{amsthm}
\usepackage{amssymb}
\usepackage{amscd}
\usepackage{amsfonts}
\usepackage{amsbsy}
\usepackage{fancyhdr}
\usepackage{graphicx}
\usepackage[all]{xy}
\usepackage{epsfig}
\usepackage{epstopdf}
\newtheorem {teo} {Theorem} [section]

\newtheorem {prop} [teo]{Proposition}
\newtheorem {cor} [teo] {Corollary}

\newtheorem{rem}[teo]{Remark}
\begin{document}
\title[Hill's approximation in a restricted four-body problem]{Hill's approximation in a restricted four-body problem}
\author[Jaime Burgos--Garc\'ia and Marian Gidea]{Jaime Burgos--Garc\'ia$^\dag$ and Marian Gidea$^\ddag$}
\address{Department of Mathematical Science, Yeshiva University, NY 10016, USA}%
\email{jbg84@xanum.uam.mx}%
\address{Department of Mathematical Science, Yeshiva University, NY 10016, USA}%
\email{Marian.Gidea@yu.edu}%
\thanks{$^\dag$ Research of M.G. was partially supported by NSF grants   DMS 0635607.}
\keywords{Four-body problem; Hill's problem; equilibrium points; stability; KAM tori; homoclinic connections; Trojan asteroids.}

\subjclass[2010]{Primary, 70F10; 70F15; Secondary 	37J15; 37J45.}

\maketitle

\begin{abstract}
We consider a restricted four-body problem on the dynamics of a massless particle under the gravitational force produced by three  mass points forming an equilateral triangle configuration. We assume that  the mass $m_3$ of one primary is  very  small compared with the other two, $m_1$ and $m_2$,  and we study the Hamiltonian system describing the motion of the massless particle in a neighborhood of $m_3$.
In a similar way to Hill's  approximation of the lunar problem,  we perform a symplectic scaling, sending the  two massive bodies   to infinity, expanding the potential  as a power series in  $m_3^{1/3}$, and taking the  limit case  when  $m_3\rightarrow 0$. We show that the limiting Hamiltonian  inherits dynamical  features from both the restricted three-body problem and the restricted four-body problem. In particular,  it  extends the classical lunar Hill problem. We investigate the geometry of the Poincar\'e sections, direct and retrograde periodic orbits about $m_3$, libration points,  periodic orbits near libration points, their  stable and unstable manifolds, and the corresponding homoclinic intersections.
The motivation for this model is the study of the motion of a satellite  near a jovian Trojan asteroid.
\end{abstract}

\section{Introduction}
One of the first explicit solutions given in the three-body problem was the Lagrange central configuration, where three bodies of different masses lie at the vertices of an equilateral triangle, with each body traveling along a specific Kepler orbit. A special case of this solution is the rigid circular motion of the three bodies, when each body moves on a circular Kepler orbit.  Such configurations can be encountered in our solar system, one of the best known examples being  the configurations consisting of the Sun, Jupiter and either one of the two families of Trojan asteroids concentrated at the Lagrangian libration points.  Other families of Trojan-like  asteroids have been observed for the Sun-Mars  and Sun-Neptune systems. Also, Saturn--Tethys--Telesto, Saturn--Tethys--Calypso, and Saturn--Dione--Helen, respectively,  are known to form  Lagragian central configurations.

In this paper we consider a  restricted four-body problem (R4BP) describing the dynamics of  a massless particle in a neighborhood of a small mass at one of the vertices of a Lagrange central configuration. Denote the masses of the primaries at the vertices of the equilateral triangle by $m_1,m_2,m_3$, with $m_3\ll m_2\leq m_1$. We consider that the motion of the massless particle occurs in a small neighborhood of $m_3$. We derive its equations of motion via the following procedure: we perform a rescaling of the coordinates  depending on $m_3^{1/3}$,  write  the associated  Hamiltonian in the rescaled coordinates as a power series in $m_3^{1/3}$, and  consider the limiting Hamiltonian obtained by letting $m_3\rightarrow 0$. This procedure provides an approximation of the motion of the massless particle in an $O(m_3^{1/3})$-neighborhood  of $m_3$, while $m_1$ and $m_2$ are sent at infinite distance through the rescaling. This model  is an extension of the classical Hill approximation of the restricted three-body problem, with the major difference that our model is a four-body problem.

One of the main advantages of the Hill approximation over the restricted four-body problem is that it allows for the analytic computation of its equilibrium points and of their stability. Also, there are additional symmetries that make the geometry much simpler. Moreover, considering realistically small values of $m_3$ in the four-body problem (e.g., corresponding to a Trojan asteroid) makes analytical studies much more difficult,  and yields technical problems with the accuracy of numerical simulations; these inconveniences are not present in the Hill approximation.

If we let  $m_2\to 0$ in our model, then the resulting limit coincides with the classical  lunar Hill problem. We recall that  G.W.~Hill developed his   lunar theory \cite{Hill} as an alternative approach for the study of the motion of the Moon around the Earth. As a first approximation, his approach considers a Kepler problem (Earth-Moon) with a gravitational perturbation produced by a far away massive body (Sun),  and assumes that  the eccentricities of the orbits of the Moon and the Earth as well as the inclination of the Moon's orbit are zero.  Hill's approximation depends on a single parameter, namely the energy of the orbit.   Through his approach Hill was able to obtain the existence of a direct, periodic orbit describing the trajectory of the Moon, and the  inclusion of orbital elements to correct it. An important remark is that this direct orbit undergoes a pitchfork  bifurcation  as the energy level is varied. Also, the classical Hill approximation has two equilibrium points (libration points) of center-saddle type.

Our Hill approximation of the restricted four-body problem depends on two parameters, the mass ratio $\mu=m_2/(m_1+m_2)$, and the energy of the system. We also observe the existence of a direct, periodic orbit that bifurcates as the energy level is varied, however the bifurcation values depends on the second parameter $\mu$. Another significant difference is that our approximation has four libration points. Two of them are of saddle-center type, as in the classical case, while the other two are of stable of center-center type for $\mu$ less than some critical value $\mu_0=\frac{1}{224}\left[112-(2(1979+37(12097)^{1/2}))^{1/2}\right]\approx 0.00898964$, and unstable of complex-saddle type otherwise. In this sense,  our model has similar characteristics to the R4BP in the case when $m_2=m_3$ is sufficiently small, which also has a libration point that changes from stable to unstable as the mass ratio $m=m_2/m_1$ is increased passed some threshold value $m^*$ (see \cite{Leandro,BurgosII}). For comparison, the planar circular restricted three-body problem has five libration   points, three of center-saddle type, and two  that are  stable for $\mu<\mu_c = \frac{1}{2}[1-(23/27)^{1/2}]\approx 0.03852$ (Routh critical value), and unstable otherwise.
Again, we stress that our model borrows features from both the  restricted three-body problem (and its Hill limit) and from the restricted four-body problem.

A concrete situations that can be modeled with our Hill approximation is the motion of a spacecraft or of a natural satellite near a Trojan asteroid.
By a  `Trojan'  we mean, in general,  an asteroid or a natural satellite that lies in a Lagrange central configuration together with the Sun and a planet, or with a  planet and a  moon.
Very well known are the Trojan asteroids of the Sun-Jupiter system, which are divided into two large families, commonly referred to as the `Trojans' and the `Greeks'. The first family is centered at a point on Jupiter's orbit around the Sun at $60^\circ$ behind the planet, and the second family is centered at a point  on the same  orbit at $60^\circ$ ahead the planet; thus each of the two points forms an equilateral triangle with the Sun and Jupiter.
These points also coincide with the Lagrangian libration points $L_5$ and $L_4$, respectively, of the Sun-Jupiter system. Astronomical observations showed that the two families  are distributed on regions that extend up to $5.2$ AU away from $L_5$ and $L_4$, and many of the asteroids  have large orbital inclinations up to $40^\circ$ from Jupiter's orbit. Empirical models for the Trojan distribution along Jupiter's orbit indicate that the locations corresponding to the maximum density coincide with $L_5$ and $L_4$  \cite{NakamuraY08}.

One possible application of our model is to design spacecraft trajectories near a Trojan asteroid.  Exploration of the Trojan asteroids was recognized by the 2013 Decadal Survey, which includes Trojan Tour and Rendezvous, among the New Frontiers missions in the decade 2013-2022.

Another possible applications is the study of the stability of the Moon-like satellite   of the Trojan asteroid 624 Hektor \cite{Marchis}; this is the
first ever discovered satellite around a  jovian Trojan asteroid.
In future works we plan to include  relevant effects produced by inclinations and librations of the asteroids or perturbations due to other bodies.

It is worth mentioning that the model proposed  in this paper is related to other types of Hill approximations.
First of all, is closely connected to the  classical lunar Hill problem, introduced in \cite{Hill}, and subsequently studied in many papers, e.g. \cite{Eckert,MeyerS,Henon,Simo,Henon2003}.
The spatial version of the problem has been studied in, e.g.,  \cite{Henon1974,Bathkin,Michalodimitrakis,GomezMM}. Some Hill approximations of the four-body problem, very different from ours, have been considered in \cite{Mohn,Scheeres1,Scheeres2}. Also, several authors, e.g.,  \cite{Cronin,PapaIII,Cecc,GabernJorba,RobutelGabernJorba,RobutelGabern}, have considered other models of restricted four-body body problems to describe the dynamics of a particle in the Sun-Jupiter-Asteroid system.

\section{The restricted four body problem}
Consider three point masses moving under
mutual Newtonian gravitational attraction in
circular periodic orbits around their center of mass, while forming an equilateral
triangle configuration. A fourth massless particle is moving under the gravitational attraction of the three mass points, without affecting their motion. This problem is known as the equilateral restricted four body problem or simply as the restricted four body problem (R4BP).  We will assume that the three masses are $m_1\geq m_2\geq m_3$, and we will refer to $m_1$ as the primary, $m_2$ as the secondary, and $m_3$ as the tertiary.

The equations of motion of the massless particle in dimensionless coordinates
relative to a synodic frame of reference that rotates together with the  point masses are:
\begin{figure}
\centering
\includegraphics[width=2.5in]{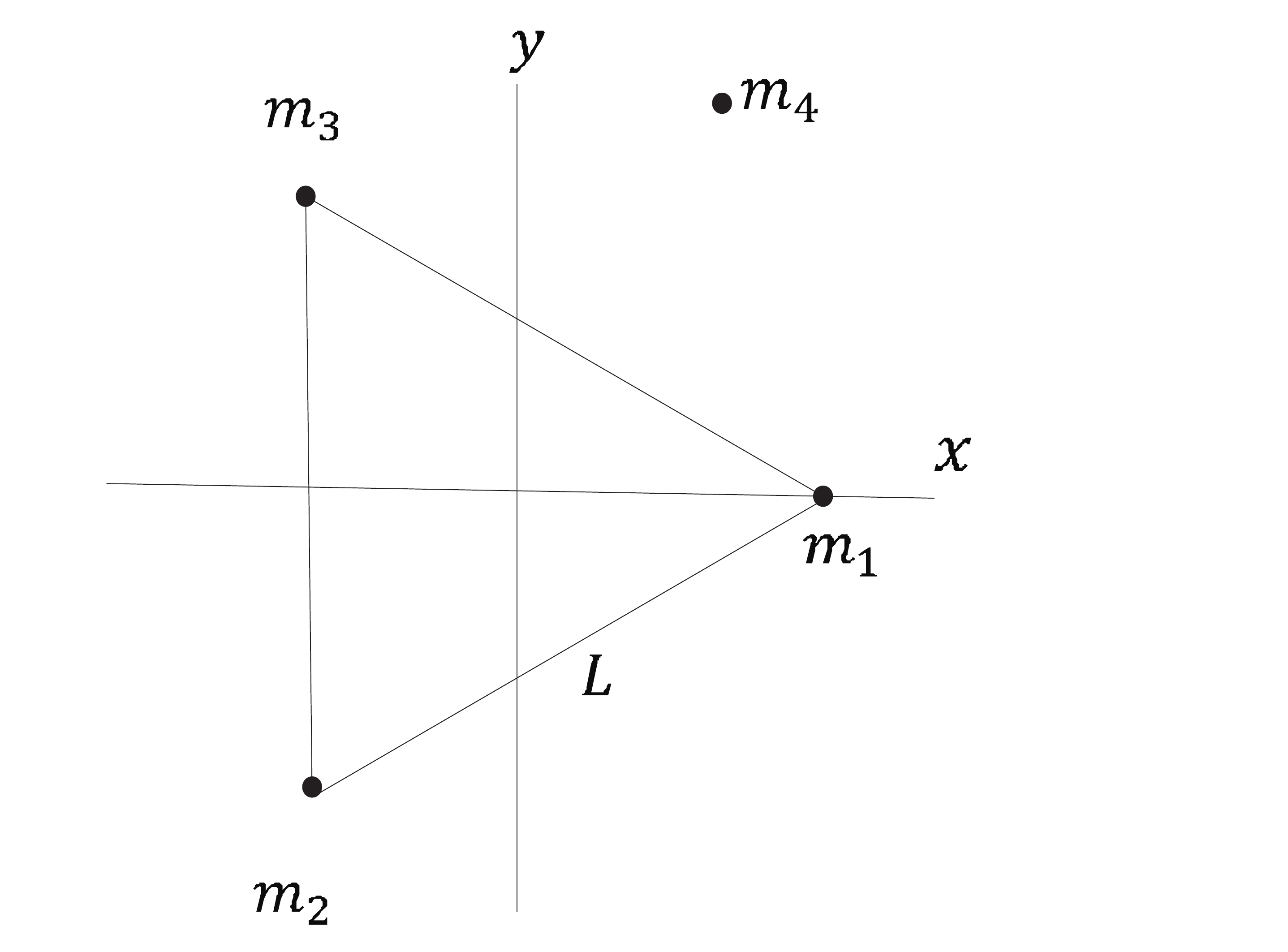}
\caption{The restricted four-body problem in a synodic system for the two equal masses case.\label{triangle}}
\end{figure}

\begin{equation}\begin{split}\label{ecuacionesfinales}
\ddot{x}-2\dot{y}&=\Omega_{x},\\
\ddot{y}+2\dot{x}&=\Omega_{y},\\
\ddot{z}&=\Omega_{z},\end{split}
\end{equation}
where $$\Omega(x,y,z)=\frac{1}{2}(x^{2}+y^{2})+\sum_{i=1}^{3}\frac{m_{i}}{r_{i}},$$
and $r_{i}=\sqrt{(x-x_{i})^{2}+(y-y_{i})^{2}+z^2}$, for $i=1,2,3$. The general expressions of the coordinates of the primaries in terms of the masses of the three point masses are given  as in \cite{PapaIII} by
\begin{equation}\begin{split}\label{coordinatesprimaries}
x_{1}&=\frac{-\vert K\vert\sqrt{m_{2}^{2}+m_{2}m_{3}+m_{3}^{2}}}{K},\\
y_{1}&=0,\\
z_1&=0,\\
x_{2}&=\frac{\vert K\vert[(m_{2}-m_{3})m_{3}+m_{1}(2m_{2}+m_{3})]}{2K\sqrt{m_{2}^{2}+m_{2}m_{3}+m_{3}^{2}}},\\
y_{2}&=\frac{-\sqrt{3}m_{3}}{2m_{2}^{3/2}}\sqrt{\frac{m_{2}^{3}}{m_{2}^{2}+m_{2}m_{3}+m_{3}^{2}}},\\
z_2&=0,\\
x_{3}&=\frac{\vert K\vert}{2\sqrt{m_{2}^{2}+m_{2}m_{3}+m_{3}^{2}}},\\ y_{3}&=\frac{\sqrt{3}}{2\sqrt{m_{2}}}\sqrt{\frac{m_{2}^{3}}{m_{2}^{2}+m_{2}m_{3}+m_{3}^{2}}},\\
z_3&=0, \end{split}
\end{equation}
where $K=m_{2}(m_{3}-m_{2})+m_{1}(m_{2}+2m_{3})$ and the three masses satisfy the relation $m_{1}+m_{2}+m_{3}=1$. It can be proved that the equations of motion have a first integral (the Jacobi integral) \[C=-(\dot{x}^2+\dot{y}^2+\dot{z}^2)+2\Omega.\]
Alternatively, we can regard the energy function $E=-C/2$ as a first integral.

We note that when  $m_{3}=0$ and $m_{2}:=\mu$ we recover the coordinates of the restricted three body problem (R3BP): \begin{equation*}\begin{split}(x_{1},y_{1},z_{1})&=(-\mu,0,0),\\(x_{2},y_{2},z_{2})&=(1-\mu,0,0),\\ (x_{3},y_{3},z_{3})&=(1/2-\mu,\sqrt{3}/2,0),\end{split}\end{equation*} where the position of the `phantom'  mass $m_{3}$ coincides with the  equilibrium point $L_{4}$ of the R3BP associated to $m_1$ and $m_2$.

Making  the transformation $\dot x=p_x+y$, $\dot y=p_y-x$, $\dot z=p_z$, the equations \eqref{ecuacionesfinales} are equivalent to the Hamiltonian equations for the Hamiltonian
\begin{equation}
H=\frac{1}{2}(p_{x}^{2}+p_{y}^{2}+p_{z}^{2})+yp_{x}-xp_{y}-\frac{m_{1}}{r_{1}}-\frac{m_{2}}{r_{2}}-\frac{m_{3}}{r_{3}},\label{originalhamiltonian}
\end{equation}
relative to the standard symplectic form $\omega=dp_x\wedge dx+dp_y\wedge dy+dp_z\wedge dz$ on $T^*W$ where $W=\mathbb{R}^3\setminus \{(x_1,y_1,z_1), (x_2,y_2,z_2), (x_3,y_3,z_3)\}$. Relative to this symplectic structure the Hamiltonian equations can be written as $\dot{\bf x}=J\nabla H({\bf x})$, where
$J=\left(
\begin{array}{cc}
0 & {\rm id} \\
{\rm -id} & 0 \\
\end{array}
\right)
$, and ${\bf x}=(x,y,z,p_x,p_y,p_z)$.
We also note that $H(x,y,z,p_x,p_y,p_z)=E(x,y,z,\dot x, \dot y,\dot z)$.

\section{The limit case and the equations of motion}

In this section we will study the limit when $m_{3}\rightarrow0$ in the Hamiltonian of the R4BP. We use a procedure similar to that in \cite{MeyerS,MeyerHDS}, by performing a symplectic scaling  depending on $m_3^{1/3}$, expanding the Hamiltonian as a power series in $m_3 ^{1/3}$ in a neighborhood of the small mass $m_{3}$, and then taking the limit as $m_3\to 0$. The resulting Hamiltonian will be a three-degree of freedom system depending on a parameter $\mu$ which becomes  equal to the mass of the secondary~$m_{2}$.

\begin{teo} \label{main theorem} After the symplectic scaling
$$(x,y,z,p_{x},p_{y},p_{z})\rightarrow m_{3}^{1/3}(x,y,z,p_{x},p_{y},p_{z}),
$$ the limit $m_{3}\rightarrow0$ of the Hamiltonian (\ref{originalhamiltonian}) restricted to a neighborhood of $m_{3}$ exists and yields a new Hamiltonian
\begin{equation}\begin{split}\label{hillhamiltonian}
H=&\frac{1}{2}(p^{2}_{x}+p^{2}_{y}+p_{z}^{2})+yp_{x}-xp_{y}+\frac{1}{8}x^2-\frac{3\sqrt{3}}{4}(1-2\mu)xy-
\frac{5}{8}y^2+\frac{1}{2}z^2\\
&-\frac{1}{\sqrt{x^2+y^2+z^2}},
\end{split}\end{equation}where $m_{1}=1-\mu$ and $m_{2}=\mu$.
\end{teo}

\begin{proof} We consider the Hamiltonian of the restricted four body problem (R4BP) in the center of mass coordinates
$$
H=\frac{1}{2}(p_{x}^{2}+p_{y}^{2}+p_{z}^{2})+yp_{x}-xp_{y}-\frac{m_{1}}{r_{1}}-\frac{m_{2}}{r_{2}}-\frac{m_{3}}{r_{3}},
$$ where $r_{i}^{2}=(x-x_{i})^2+(y-y_{i})^2+z^2$ and $(x_{i},y_{i})$ denotes the $xy$-coordinates of the primary $m_{i}$ for $i=1,2,3$. We make the change of coordinates $x\rightarrow x+x_{3}$, $y\rightarrow y+y_{3}$, $z\rightarrow z$, $p_{x}\rightarrow p_{x}-y_{3}$, $p_{y}\rightarrow p_{y}+x_{3}$, $p_{z}\rightarrow p_{z}$, therefore in these new coordinates the Hamiltonian (\ref{originalhamiltonian}) becomes

\begin{equation}
H=\frac{1}{2}(p_{x}^{2}+p_{y}^{2}+p_{z}^{2})+yp_{x}-xp_{y}-(x_{3}x+y_{3}y)-\frac{m_{1}}{\bar{r}_{1}}-\frac{m_{2}}{\bar{r}_{2}}-\frac{m_{3}}{\bar{r}_{3}},\label{traslatedhamiltonian}
\end{equation}
where  $\bar{r}_{i}^{2}=(x+x_{3}-x_{i})^2+(y+y_{3}-y_{i})^2+z^2:=(x+\bar{x}_{i})^2+(y+\bar{y}_{i})^2+z^2$,  for $i=1,2,3$. We expand the terms $\frac{1}{\bar{r}_{1}}$ and $\frac{1}{\bar{r}_{2}}$ in Taylor series around the new origin of coordinates; if we ignore the constant terms we obtain the following expressions \begin{equation*}\begin{split}f^{1}:&=\frac{1}{\bar{r}_{1}}=\sum_{k\ge1}P_{k}^{1}(x,y,z),\\
f^{2}:&=\frac{1}{\bar{r}_{2}}=\sum_{k\ge1}P_{k}^{2}(x,y,z),\end{split}\end{equation*} where $P_{k}^{j}(x,y,z)$ is a homogenous polynomial of degree $k$ for $j=1,2$. We perform  the following symplectic scaling $x\rightarrow m_{3}^{1/3}x$, $y\rightarrow m_{3}^{1/3}y$, $z\rightarrow m_{3}^{1/3}z$, $p_{x}\rightarrow m_{3}^{1/3}p_{x}$, $p_{y}\rightarrow m_{3}^{1/3}p_{y}$ $p_{z}\rightarrow m_{3}^{1/3}p_{z}$ with multiplier $m_{3}^{-2/3}$, obtaining

\begin{equation}\begin{split}\label{scaled}
H=&\frac{1}{2}(p_{x}^{2}+p_{y}^{2}+p_{z}^{2})+yp_{x}-xp_{y}-\frac{1}{\bar{r}_{3}}-
m_{3}^{-1/3}(x_{3}x+y_{3}y+P_{1}^{1}+P_{1}^{2})\\
&-\sum_{k\ge2}m_{3}^{\frac{k-2}{3}}m_{1}P_{k}^{1}(x,y,z)-\sum_{k\ge2}m_{3}^{\frac{k-2}{3}}m_{2}P_{k}^{2}(x,y,z).
\end{split}\end{equation}
A straightforward computation shows \begin{equation*}\begin{split}P_{1}^{1}&=m_{1}(\frac{x_{1}-x_{3}}{\bar{r}_{1}^{3}}x+\frac{y_{1}-y_{3}}{\bar{r}_{1}^{3}}y), \\P_{1}^{2}&=m_{2}(\frac{x_{2}-x_{3}}{\bar{r}_{2}^{3}}x+\frac{y_{2}-y_{3}}{\bar{r}_{2}^{3}}y),\end{split}\end{equation*} where $\bar{r_{i}}=\sqrt{(x_{3}-x_{i})^2+(y_{3}-y_{i})^2}$.
It is important to note that the first partial derivative with respect to the variable $z$ is given by $$f^{i}_{z}=-\frac{z}{\bar{r}^{3}_{i}},$$ for $i=1,2$. Therefore we obtain
\begin{eqnarray*}
f^{i}_{z}(0,0,0)=f^{i}_{xz}(0,0,0)=f^{i}_{yz}(0,0,0)=0, \label{secondnullderivatives}
\end{eqnarray*}
and $$f^{i}_{zz}(0,0,0)=-1.$$ Now if we recall that the three masses are in equilateral configuration with length equal to one and we use the relation $m_{1}=1-m_{2}-m_{3}$ we obtain \begin{equation*}\begin{split}m_{3}^{-1/3}(x_{3}x+y_{3}y+P_{1}^{1}+P_{1}^{2})
=&m_{3}^{-1/3}[x_{1}+m_{2}(x_{2}-x_{1})-m_{3}(x_{1}-x_{3})]x,\\ &-m_{3}^{-1/3}[y_{1}+m_{2}(y_{2}-y_{1})-m_{3}(y_{1}-y_{3})]y,\end{split}\end{equation*}
which, in terms of the coordinates of the point masses  (\ref{coordinatesprimaries})  we can write as $$m_{3}^{-1/3}[y_{1}+m_{2}(y_{2}-y_{1})-m_{3}(y_{1}-y_{3})]=-m_{3}^{2/3}m_{2}s_{1}(m_{1},m_{2},m_{3})+m_{3}^{2/3}y_{3},$$ where $$s_{1}(m_{1},m_{2},m_{3})=\sqrt{\frac{3m_{2}^{3}}{4m_{2}^{3}(m_{2}^{2}+m_{2}m_{3}+m_{3}^{2})}},$$ and $$y_{3}=\frac{\sqrt{3}}{2\sqrt{m_{2}}}\sqrt{\frac{m_{2}^{3}}{m_{2}^{2}+m_{2}m_{3}+m_{3}^{2}}}.$$
A similar computation shows that the coefficient $m_{3}^{-1/3}[x_{1}+m_{2}(x_{2}-x_{1})-m_{3}(x_{1}-x_{3})]$ can be written in terms of a positive power of $m_{3}$. Therefore, the Hamiltonian (\ref{scaled}) becomes
\begin{equation}
H=\frac{1}{2}(p_{x}^{2}+p_{y}^{2}+p_{z}^{2})+yp_{x}-xp_{y}-\frac{1}{r}-m_{1}P_{2}^{1}-m_{2}P_{2}^{2}+\mathcal{O}(m_{3}^{1/3}).\label{unlimitedhamiltonian}
\end{equation}
We have defined $r=\bar{r}_{3}$.

Now we  take the limit $m_{3}\rightarrow0$ in the expression (\ref{unlimitedhamiltonian}); this means that the primary and the secondary are sent at an infinite distance, and their total mass becomes infinite.  After some computations the limiting Hamiltonian becomes
\begin{equation}\begin{split}\label{hillhamiltonian}
H=&\frac{1}{2}(p^{2}_{x}+p^{2}_{y}+p_{z}^{2})+yp_{x}-xp_{y}+\frac{1}{8}x^2-\frac{3\sqrt{3}}{4}(1-2\mu)xy-
\frac{5}{8}y^2+\frac{1}{2}z^2\\
&-\frac{1}{\sqrt{x^2+y^2+z^2}},
\end{split}\end{equation}
where $m_{2}:=\mu$ and $m_{1}=1-\mu$.\end{proof}

The gravitational and effective potential corresponding to the Hamiltonian \eqref{hillhamiltonian} are:
\begin{equation}
U=-\frac{1}{8}x^2+\frac{3\sqrt{3}}{4}(1-2\mu)xy+\frac{5}{8}y^2-\frac{1}{2}z^2+\frac{1}{\sqrt{x^2+y^2+z^2}},
\label{hillgravpotential}
\end{equation}
\begin{equation}
\Omega=\frac{1}{2}(x^2+y^2)+U=\frac{3}{8}x^2+\frac{3\sqrt{3}}{4}(1-2\mu)xy+\frac{9}{8}y^2-\frac{1}{2}z^2+\frac{1}{\sqrt{x^2+y^2+z^2}},\label{hillefectivepotential}
\end{equation}
respectively. The equations of motion can be written as in the full problem
\begin{equation}
\label{hillfinalequations}
\ddot{x}-2\dot{y}=\Omega_{x},
\end{equation}
$$\ddot{y}+2\dot{x}=\Omega_{y},$$ $$\ddot{z}=\Omega_{z},$$
with $\Omega$ is given by the equation (\ref{hillefectivepotential}).
\begin{rem}{$ $}\label{rem1}

\begin{itemize}
\item The expression
\begin{equation}
Q=\frac{1}{2}(p^{2}_{x}+p^{2}_{y}+p_{z}^{2})+yp_{x}-xp_{y}+\frac{1}{8}x^2-\frac{3\sqrt{3}}{4}(1-2\mu)xy-\frac{5}{8}y^2+\frac{1}{2}z^2,\label{quadraticpart}
\end{equation}
is the quadratic part of the expansion of the Hamiltonian of the restricted three-body problem centered at the Lagrange libration point $L_{4}$.
\item The range of the mass parameter is $\mu\in[0,1/2]$. The special case $\mu=0$  coincides with the classical lunar Hill problem after some coordinate transformation  (see Section \ref{transformation}). The case where $\mu=1/2$ corresponds to the case of equally massive bodies, similar to binary star systems.
\item We will prove in Section \ref{equilibriumpoints} that the system \eqref{hillhamiltonian} has 4 equilibrium points in a neighborhood of the tertiary, and these equilibrium points  possess the same stability properties as in the full R4BP when $m_{3}$ is sufficiently small but non zero.
\end{itemize}\end{rem}

In Fig. \ref{limithillregions} we plot the Hill regions (i.e., the projections of the energy manifold onto the configuration space) for the planar problem $z=0$, when $\mu= 0.00095$, which corresponds to mass ratio of the Sun-Jupiter system. In the first row we show side by side the Hill  regions for the limit problem and for the full R4BP when $m_{3}=7.03\times10^{-12}$, the mass ratio of the asteroid 624 Hektor; the lines in the second figure are imaginary lines connecting $m_{3}$ with the remaining masses. In the second row we plot the position of the tertiary and its relation with the primary and secondary in the full R4BP. In the third row we plot  the positions of the four equilibrium points are around the tertiary.

\begin{figure}
  \centering
\begin{tabular}{cc}
  \includegraphics[width=2.35in]{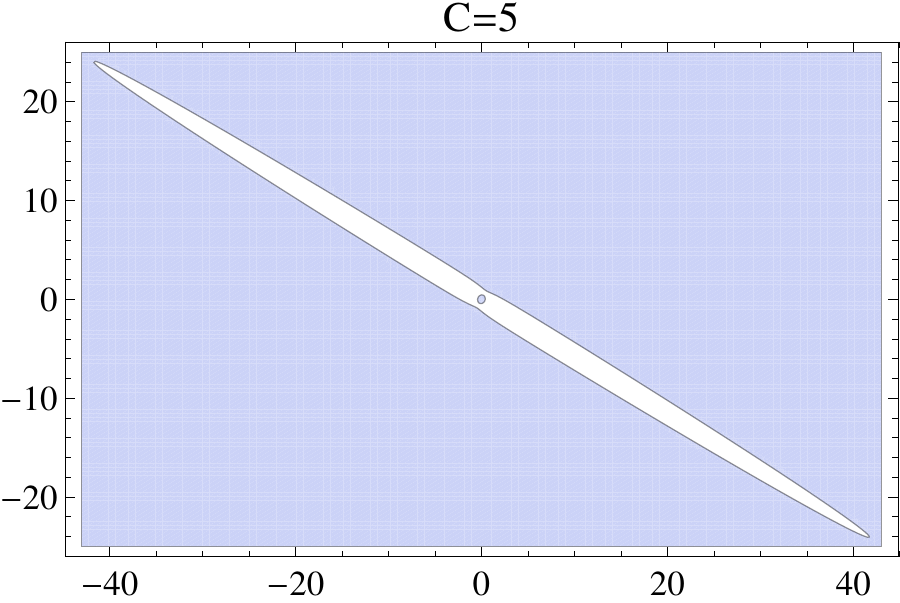}& \includegraphics[width=2.35in,height=1.56in]{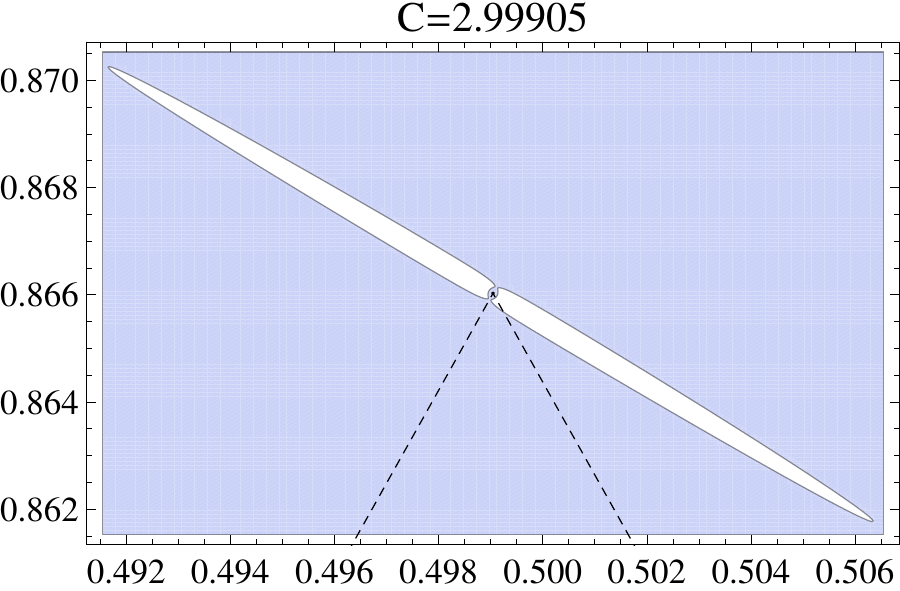}\\
  \includegraphics[width=2.35in]{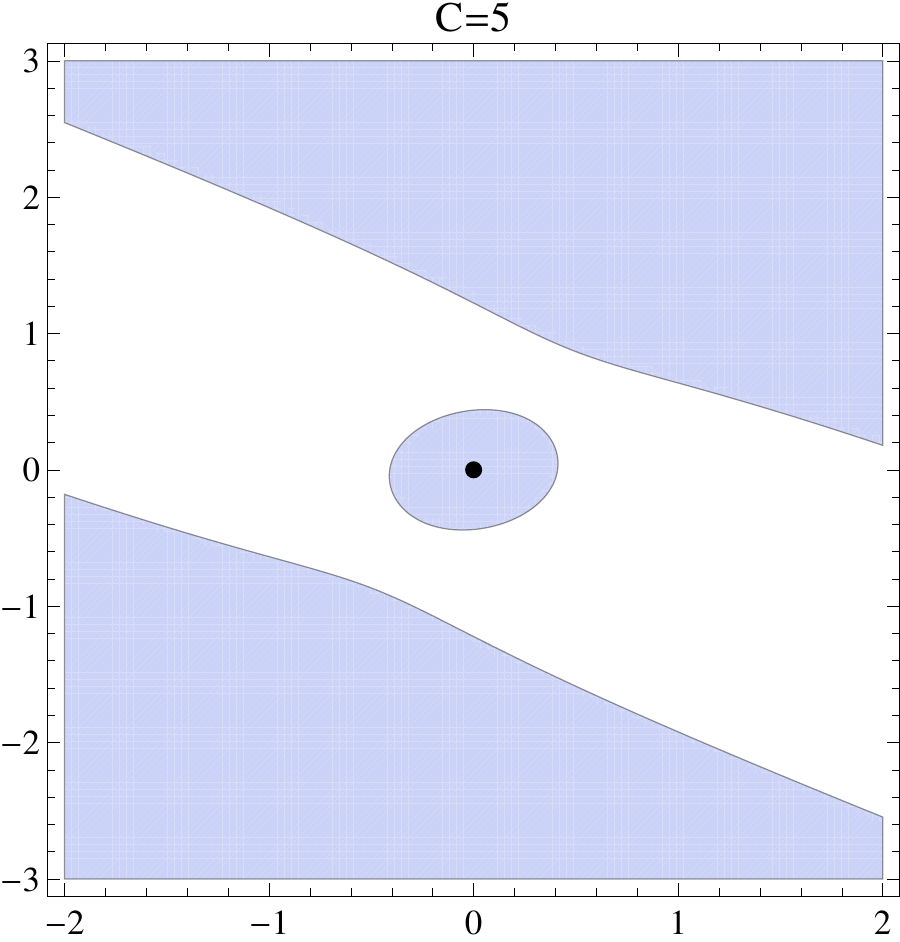}&\includegraphics[width=2.35in,height=2.48in]{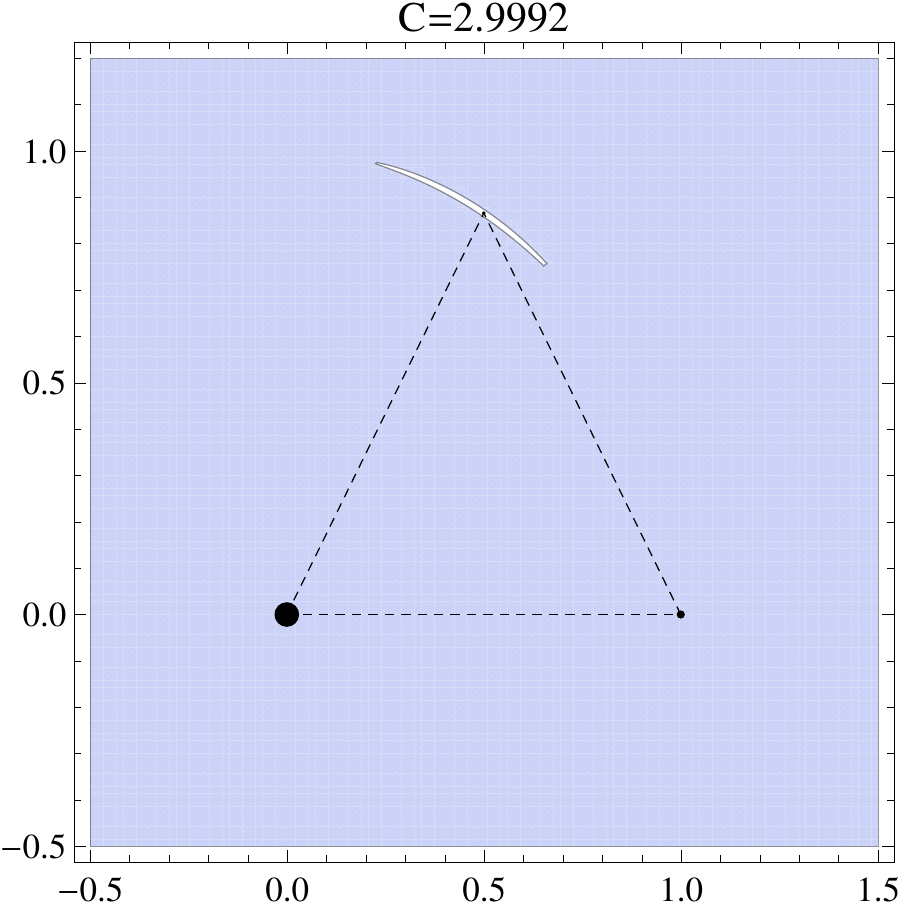}\\
   \includegraphics[width=2.35in]{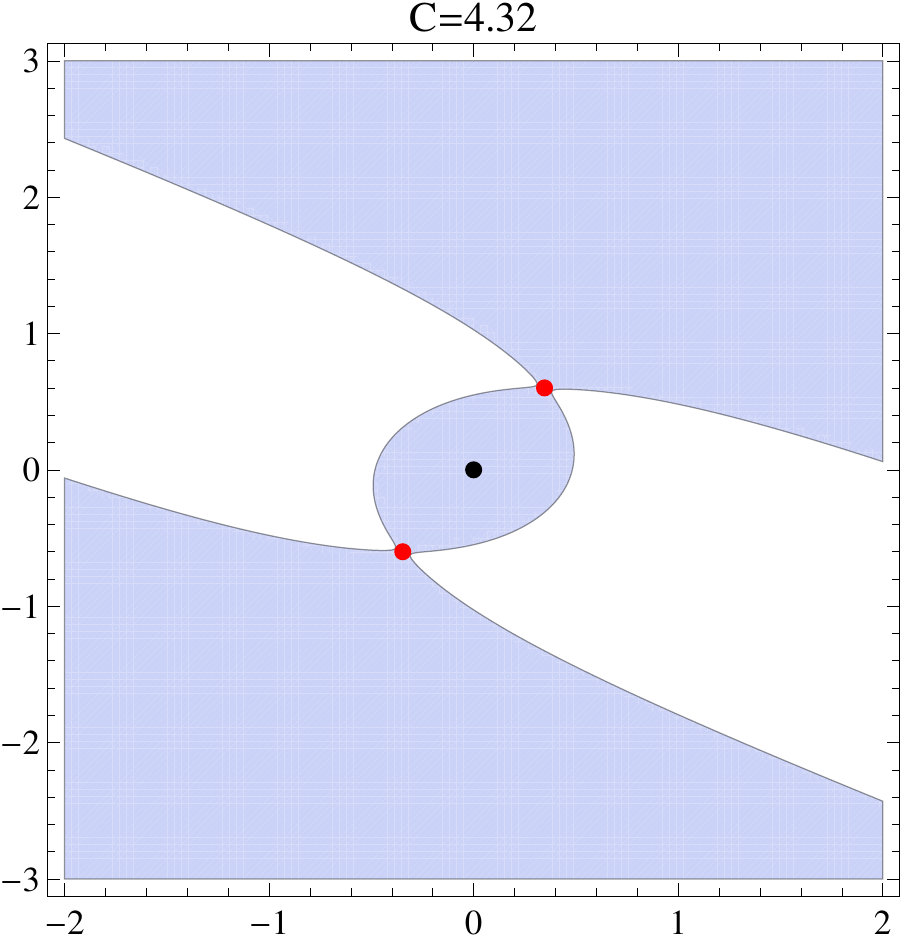}& \includegraphics[width=2.35in,height=2.48in]{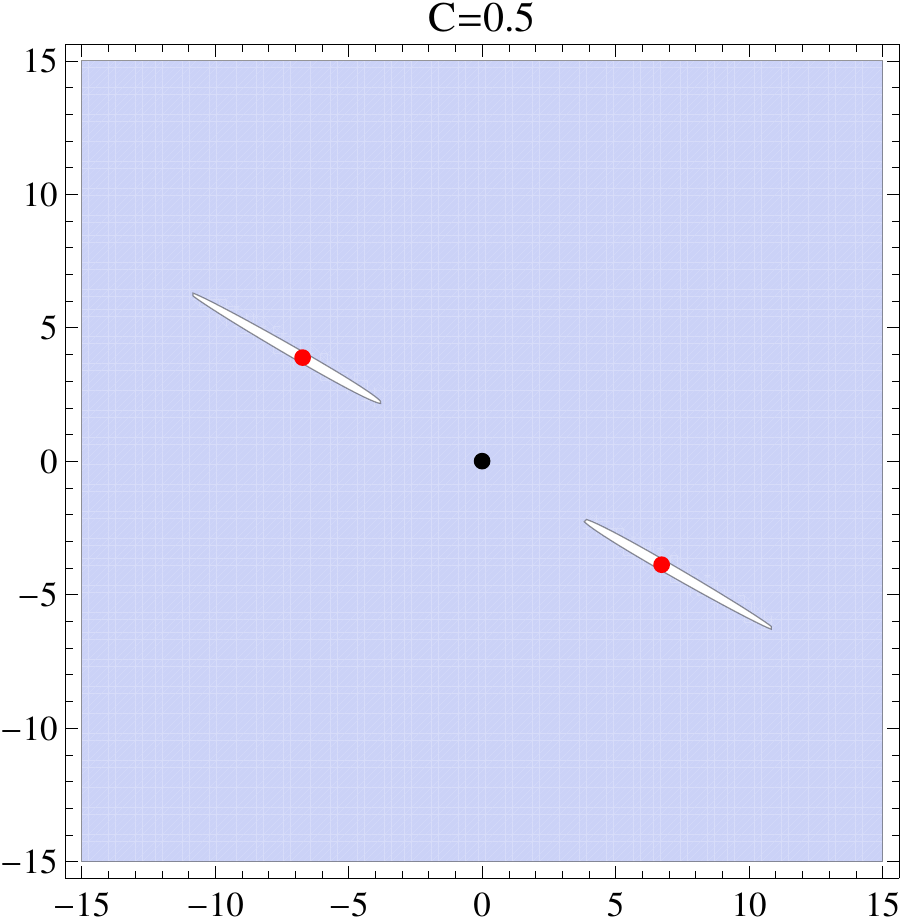}
\end{tabular}
 \caption{Hill's regions (blue areas) for $\mu= 0.00095$. First row, left to right. Hill's region for a the limit problem. Hill's region for a the full R4BP when $m_{3}=7.03\times10^{-12}$ (mass of 624 Hektor). Second row, left to right. Magnification of the first figure for the limit case and the location of the Hill's region in the R4BP. Third row.  Position of the equilibrium points (red dots) for the limit case.}\label{limithillregions}
\end{figure}

\subsection{Transformation of the equations of motion in the planar case}
\label{transformation}
As it was pointed in Remark \ref{rem1}, when we let $\mu=0$ in the Hamiltonian (\ref{hillhamiltonian}) we should  recover the Hamiltonian of the classical Hill lunar problem,  although this is not evident at first sight. However,  after applying a rotation in the $xy$-plane we obtain a new Hamiltonian with much nicer properties.

\begin{cor} \label{main corollary} The  system of equations (\ref{hillfinalequations}) is equivalent, via a rotation, with the system
\begin{equation}\begin{split}\label{finalequations}
\ddot{\bar{x}}-2\dot{\bar{y}}&=\Omega_{\bar{x}},\\
\ddot{\bar{y}}+2\dot{\bar{x}}&=\Omega_{\bar{y}},\\
\ddot{\bar{z}}&=\Omega_{\bar{z}},\end{split}\end{equation}
with
\begin{equation}\label{rotatedeffective}
\bar{\Omega}=\frac{1}{2}(\lambda_{2}\bar{x}^2+\lambda_{1}\bar{y}^2-\bar{z}^2)+
\frac{1}{\sqrt{\bar{x}^2+\bar{y}^2+\bar{z}^2}},
\end{equation}
where $\lambda_{2}$ and $\lambda_{1}$ are the eigenvalues corresponding to the rotation transformation in the $xy$-plane.
\end{cor}

\textit{Proof}. Because of the rotation is performed in the plane, we restrict the computations to the planar case. The planar effective potential restricted to the $xy$-plane is given by the expression $$\Omega=\frac{3}{8}x^2+\frac{3\sqrt{3}}{4}(1-2\mu)xy+\frac{9}{8}y^2+\frac{1}{\sqrt{x^2+y^2}},$$ which, rewritten  in  matrix notation, becomes
\begin{equation}
\Omega=\frac{1}{2}z^{T}Mz+\frac{1}{\Vert z\Vert},
\end{equation}
where $z=(x,y)^{T}$ and  $$M=\left(\begin{array}{cc}\frac{3}{4}  & \frac{3\sqrt{3}}{4}(1-2\mu) \\\frac{3\sqrt{3}}{4}(1-2\mu)  & \frac{9}{4}\end{array}\right).$$ We notice that the matrix $M$ is symmetric, so its eigenvalues are real, the corresponding eigenvectors $v_{1}$ and $v_{2}$ are orthogonal, and the corresponding orthogonal matrix $C=col(v_{2},v_{1})$ is an isometry. We recall that a matrix is orthogonal if $C^{-1}=C^{T}$. The matrix $M$ has eigenvalues \begin{eqnarray*}\lambda_{1}&=\frac{3}{2}(1-d),\\\lambda_{2}&=\frac{3}{2}(1+d),\end{eqnarray*} with corresponding eigenvectors \begin{eqnarray*}v_{1}&=\left(\frac{1+2d}{(2\mu-1)\sqrt{3+(\frac{1+2d}{1-2\mu})^2}},\frac{\sqrt{3}}{\sqrt{3+(\frac{1+2d}{1-2\mu})^2}}\right),\\
v_{2}&=\left(\frac{1-2d}{(2\mu-1)\sqrt{3+(\frac{1-2d}{1-2\mu})^2}},\frac{\sqrt{3}}{\sqrt{3+(\frac{1-2d}{1-2\mu})^2}}\right),
\end{eqnarray*} where $d=\sqrt{1-3\mu+3\mu^2}$. The eigenvectors have been chosen such that $\Vert v_{1}\Vert=\Vert v_{2}\Vert=1$.  We notice that the above expressions are singular when $\mu=1/2$, however such expressions are no necessary for this case because the matrix $M$ is already diagonal, therefore the eigenvectors are not needed for the transformation. The equations of motion for the planar case can be written as
\begin{equation}
\ddot{z}-2J_{2}\dot{z}=Mz-\frac{z}{\Vert z\Vert^3}=0,\label{gradient}
\end{equation}

where $$J_{2}=\left(\begin{array}{cc} 0  & 1 \\ -1  & 0\end{array}\right).$$

Now we consider the linear change of variables $z=Cw$ with $w=(\bar{x},\bar{y})^{T}$, we substitute in the equation (\ref{gradient}) and multiply by $C^{-1}$ by the left, in this way we obtain
\begin{equation}
\ddot{w}-2C^{-1}J_{2}C\dot{w}=C^{-1}MCw-\frac{C^{-1}Cw}{\Vert Cw\Vert^3}=0.\label{gradientdiagonal}
\end{equation}
It is easy to see that $D=C^{-1}MC$, where $D$ is given by the diagonal matrix $$D=\left(\begin{array}{cc}\lambda_{2}  & 0 \\0 &\lambda_{1}\end{array}\right),$$ and $\Vert Cw\Vert^3=\Vert w\Vert^3$ because $C$ is a isometry. Therefore the equation (\ref{gradientdiagonal}) becomes $$\ddot{w}-2C^{-1}J_{2}C\dot{w}=Dw-\frac{w}{\Vert w\Vert^3}=0.$$ Now, if we denote $v_{1}=(v_{11},v_{12})^{T}$ and $v_{2}=(v_{21},v_{22})^{T}$, after some computations we obtain
$$
C^{-1}J_{2}C=\left(\begin{array}{cc}0  & a \\-a &0\end{array}\right),
$$
where $a=v_{21}v_{12}-v_{22}v_{11}$. Because of $v_{1}\cdot v_{2}=0$, we obtain $v_{21}=-\frac{v_{12}v_{22}}{v_{11}}$ so we can write $a=-\frac{v_{22}}{v_{11}}\Vert v_{1}\Vert^{2}$, but $\Vert v_{1}\Vert=1$, so $$a=-\frac{v_{22}}{v_{11}}.$$ It is easy to see that $v_{11}\ne0$ for $\mu\in[0,1/2)$. Therefore the equation (\ref{gradientdiagonal}) becomes

$$
\ddot{w}-2aJ_{2}\dot{w}=Dw-\frac{w}{\Vert w\Vert^3}=0.
$$
It is not difficult to see that the coefficient $a$ becomes
$$a=\frac{\sqrt{3}(1-2\mu)}{1+2d}\left(-\frac{(1-2\mu)^2}{1-4d^2}\right)^{1/2}=\sqrt{3}\left(\frac{(1-2\mu)^2}{3(1-2\mu)^2}\right)^{1/2}=1.
$$

Therefore, the change of coordinates is symplectic. For each $\mu\in[0,1/2)$ we obtain the equations
\begin{equation}\begin{split}\label{finalequations}
\ddot{\bar{x}}-2\dot{\bar{y}}=\bar{\Omega}_{\bar{x}},\\
\ddot{\bar{y}}+2\dot{\bar{x}}=\bar{\Omega}_{\bar{y}},\end{split}
\end{equation}
with
\begin{equation}
\bar{\Omega}=\frac{1}{2}(\lambda_{2}\bar{x}^2+\lambda_{1}\bar{y}^2)+\frac{1}{\Vert w\Vert}.\label{rotatedpotential}
\end{equation}

For the special case when $\mu=0$, the eigenvalues and eigenvectors are \begin{equation*}\begin{split}\lambda_{1}=0,\\ \lambda_{2}=3,\end{split}\end{equation*} with respective eigenvectors
\begin{equation*}\begin{split}v_{1}&=\left(-{\sqrt{3}}/{2},{1}/{2}\right)\\ v_{2}&=\left({1}/{2},\sqrt{3}/{2}\right).
\end{split}\end{equation*}
It is easy to see that the matrix $C$ is a rotation with angle $\pi/3$ when $\mu=0$. Therefore the equations (\ref{finalequations}) take  the form
\begin{equation*}\begin{split}
\ddot{\bar{x}}-2\dot{\bar{y}}&=\bar{\Omega}_{\bar{x}},\\
\ddot{\bar{y}}+2\dot{\bar{x}}&=\bar{\Omega}_{\bar{y}},\end{split}\end{equation*}
with $$\bar{\Omega}=\frac{3}{2}\bar{x}^2+\frac{1}{\Vert w\Vert}$$ exactly as in the classical Hill problem. \qed

Therefore   the above system is an extension of the classical Hill  lunar problem. In order to simplify the notation, we are going to omit the bars for $x$ and $y$. From the expressions for $\Omega_{x}$ and $\Omega_{y}$ we can notice the following properties
\begin{eqnarray*}
\Omega_{x}(x,-y)&=&  \Omega_{x}(x,y),\\
\Omega_{y}(x,-y)&=&- \Omega_{x}(x,y).
\end{eqnarray*}
Using these properties it easy to see that the equations (\ref{finalequations}) are invariant under the transformations $x \rightarrow x$, $y \rightarrow -y$, $\dot{x} \rightarrow -\dot{x}$, $\dot{y} \rightarrow \dot{y}$, $\ddot{x} \rightarrow \ddot{x}$, $\ddot{y} \rightarrow -\ddot{y}$ as a consequence we have the well known symmetry respect the $x$-axis. In fact a similar argument shows that the equations (\ref{finalequations}) are also symmetric respect the $y$-axis.

Now we conclude  that the Hamiltonian in these new coordinates is given by the expression
\begin{equation}\begin{split}
H(x,y,z,p_x,p_y,p_z)=&\frac{1}{2}(p_x^2+p_y^2+p_z^2)+yp_x-xp_y+ax^2+by^2+cz^2\\&-\frac{1}{\sqrt{x^2+y^2+z^2}},
\label{rotatedspatialhamiltonian}
\end{split}
\end{equation}
where $a=(1-\lambda_{2})/2$, $b=(1-\lambda_{1})/2$ and $c=1/2$.

\section{The equilibrium points of the system.}
\label{equilibriumpoints}
\subsection{Computation of the equilibrium points.}
In the previous section we saw that the special case $\mu=0$ corresponds exactly to the classical Hill problem. It is well known that such problem possesses two saddle-center type equilibrium points \cite{Sz}, so in this section we will focus in the case when $\mu\in(0,1/2]$. We will prove that the system has 4 equilibrium points that can be computed explicitly in terms of the mass parameter $\mu$. In order to find the equilibrium points of the limit case, as usual, we need to find the critical points of the effective potential \eqref{rotatedeffective}; an easy computation shows that \[\Omega_{z}=-z(1+\frac{1}{r^3}),\] the equation $\Omega_{z}=0$ implies that $z=0$ so the equilibrium points of the system are coplanar. Therefore, it is enough to study the critical points of the planar effective potential
\[\Omega=\frac{1}{2}(\lambda_{2}x^2+\lambda_{1}y^2)+\frac{1}{\sqrt{x^2+y^2}}.\]
After computing the first partial derivatives we have to solve the equations
\begin{equation*}\begin{split}
\Omega_{x}=&\left(\lambda_{2}-\frac{1}{(x^2+y^2)^{3/2}}\right)x=0,\\
\Omega_{y}=&\left (\lambda_{1}-\frac{1}{(x^2+y^2)^{3/2}}\right)y=0.
\end{split}
\end{equation*}

The case $x=y=0$ corresponds to a singularity and the case $x\ne0$, $y\ne0$ yields a contradiction. Therefore when $y=0$ we have $(x^2)^{3/2}=\lambda^{-1}_{2}$ or equivalently $$\vert x\vert=\frac{1}{\sqrt[3]{\lambda_{2}}},$$
on the other hand, when $x=0$ we have $(y^2)^{3/2}=\lambda^{-1}_{1}$  or equivalently $$\vert y\vert=\frac{1}{\sqrt[3]{\lambda_{1}}},$$
therefore we obtain four equilibrium points given by
\begin{eqnarray*}
L_{1}=\left(\frac{1}{\sqrt[3]{\lambda_{2}}},0\right),
L_{2}=\left(-\frac{1}{\sqrt[3]{\lambda_{2}}},0\right),
L_{3}=\left(0,\frac{1}{\sqrt[3]{\lambda_{1}}}\right),
L_{4}=\left(0,-\frac{1}{\sqrt[3]{\lambda_{1}}}\right),
\end{eqnarray*}

We remark that the presence of a second massive body perturbing the system produces two additional equilibrium points to those of the classical Hill problem. When $\mu\to 0$, $\lambda_1\to 0$ hence $L_{3},L_{4}$ are sent to infinity.
The stability of  $L_{3}$ and $L_{4}$ depends on the value of the parameter $\mu$ as we will see in the next section.

\subsection{Study of the stability of the equilibrium points.}
In the previous subsection we obtained explicit expressions of the four equilibrium points in terms of the parameter $\mu$, so  we can analyze the linear stability in the whole range $\mu\in[0,1/2]$. We will perform such analysis for the
planar case $z=0$. As usual, we need to study the linear system $\boldmath{\dot{\xi}}=\bf{A}\bf{\xi}$, where $\bf{\xi}\rm=(x,y,\dot{x},\dot{y})^{T}$ and $A$ is the matrix

\begin{equation}
\left(\begin{array}{cccc}0 & 0 & 1 & 0 \\0 & 0 & 0 & 1 \\\Omega_{xx} & \Omega_{xy} & 0 & 2 \\\Omega_{xy} & \Omega_{yy} & -2 & 0\end{array}\right)\label{matrixlinearization}
\end{equation}
where the partial derivatives are given by the expressions \begin{equation*}\begin{split}\Omega_{xx}&=\lambda_{2}+\frac{3x^2}{(x^2+y^2)^{5/2}}-\frac{1}{(x^2+y^2)^{3/2}},\\ \Omega_{yy}&=\lambda_{1}+\frac{3y^2}{(x^2+y^2)^{5/2}}-\frac{1}{(x^2+y^2)^{3/2}},\\ \Omega_{xy}&=\frac{3xy}{(x^2+y^2)^{5/2}},\end{split}\end{equation*} and they need to be evaluated at each $L_{i}$ for $i=1,2,3,4.$ Because of the symmetries of the equilibrium points, we just need to study the equilibrium points $L_{1}$ and $L_{3}$. The characteristic polynomial of the matrix (\ref{matrixlinearization}) is given by the expression

\begin{equation}
p(\lambda)=\lambda^4+A\lambda^2+B,
\end{equation}
where $A=4-\Omega_{xx}-\Omega_{yy}$, $B=\Omega_{xx}\Omega_{yy}-\Omega_{xy}^2$. Therefore the four eigenvalues of the matrix (\ref{matrixlinearization}) at each point are given by  $$\Lambda_{1,2,3,4}=\pm\frac{1}{\sqrt{2}}\sqrt{-A\pm\sqrt{D}},$$ with $D=A^2-4B$. In the Fig. \ref{stabilityL1} we can observe the behavior of  $A$, $B$ and $D$ as functions of the parameter $\mu$.

\begin{figure}
  \centering
  \begin{tabular}{cc}
  \includegraphics[width=2.35in]{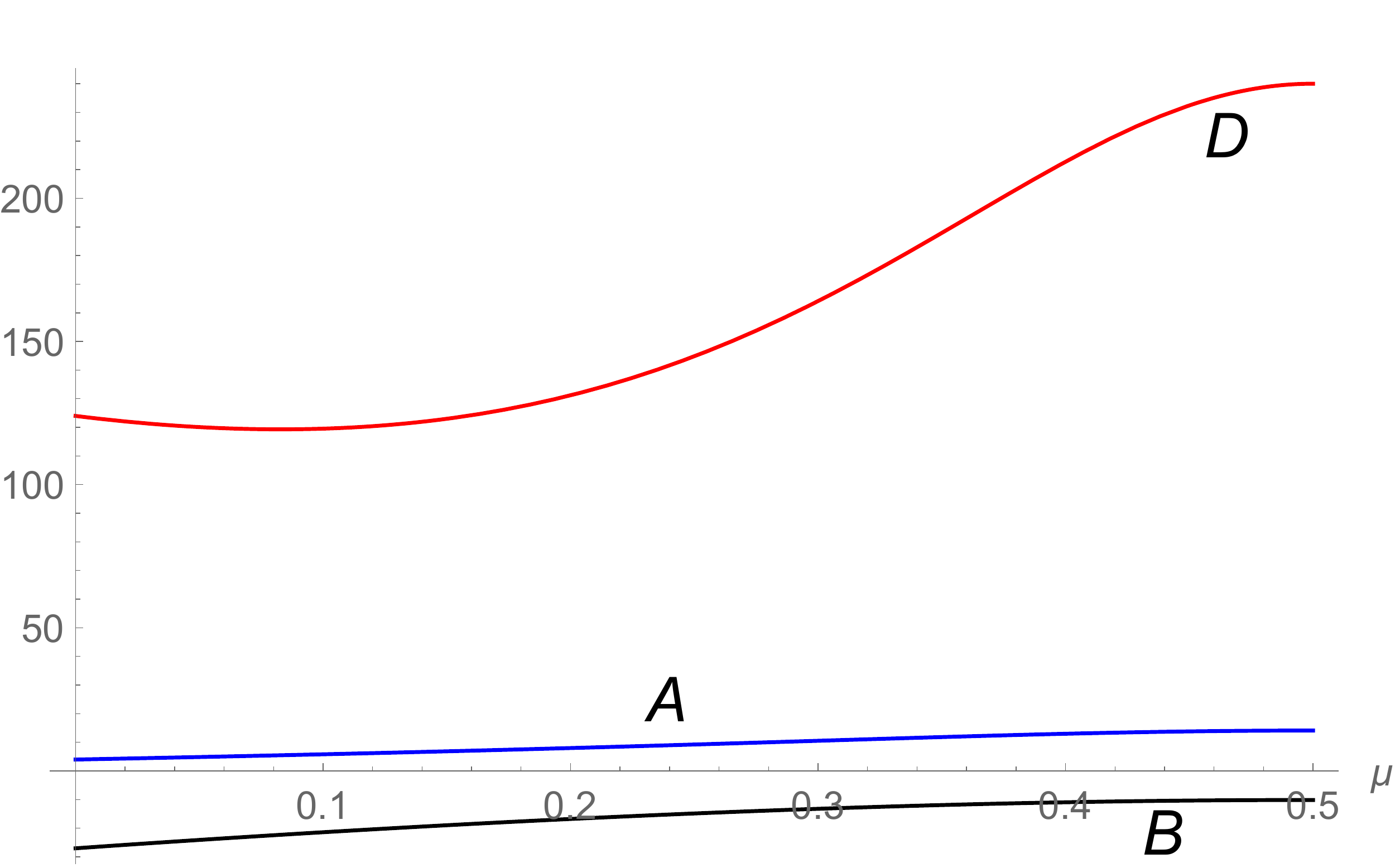}&
  \includegraphics[width=2.35in]{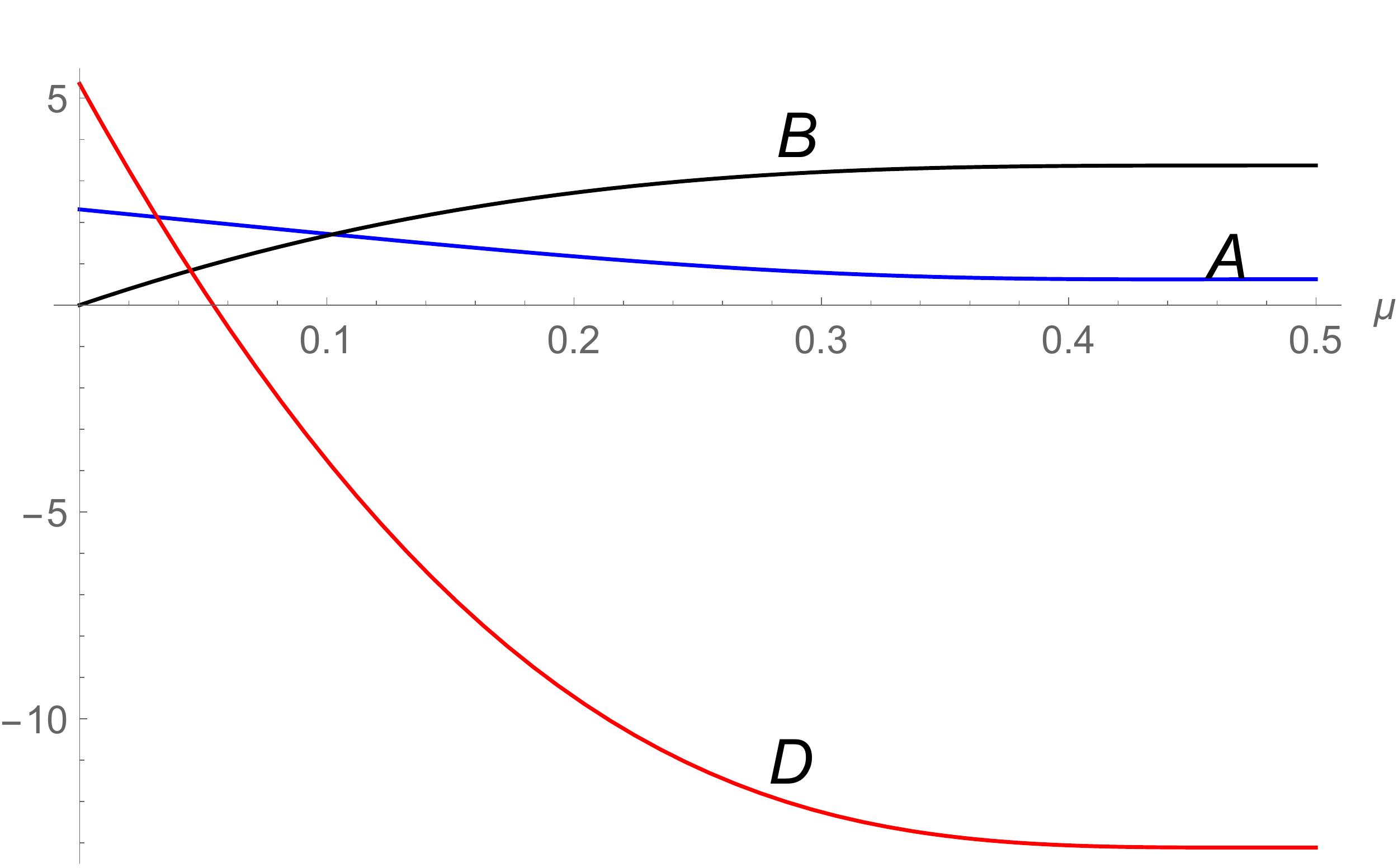}
  \end{tabular}
 \caption{Left. The coefficients $A$ (in blue), $B$ (in black) and $D$ (in red) for the equilibrium point $L_{1}$ as functions of the mass parameter $\mu$. Right. The coefficients $A$, $B$ and $D$ for the equilibrium point $L_{3}$.}
\label{stabilityL1}
\end{figure}

A equilibrium point is linearly stable if only if  $A$, $B$ and $D$ are non-negative. For the points $L_{1}$ and $L_{2}$ we have
\begin{equation*}\begin{split}\Omega_{xx}&=3\lambda_{2},\\
\Omega_{yy}&=\lambda_{1}-\lambda_{2},\\\Omega_{xy}&=0.\end{split}\end{equation*} For this case the coefficient $B$ is given by $B=-\frac{27}{2}(1+d)d$. It can be easily verified that $d=d(\mu)$ is a decreasing function between $1/2\le d\leq1$ when $\mu\in[0,1/2]$. Therefore the coefficient is always negative and consequently the equilibrium points $L_{1}$ and $L_{2}$ are unstable, in fact, if we apply the argument shown on page 33 of \cite{Burgosthesis},
which is equivalent to analyze the sign of the discriminant $D$ we obtain the following:

\begin{prop} The coefficient $B$ is negative for $\mu\in[0,1/2]$ so the equilibrium points $L_{1}$ and $L_{2}$ are unstable for this range of values of the mass parameter, in fact, the eigenvalues are given by $\pm\Lambda$ and $\pm\textit{i}\omega$ with $\Lambda>0$ and  $\omega>0$.
\end{prop}

For the points $L_{3}$ and $L_{4}$ we have
\begin{equation*}\begin{split}\Omega_{xx}&=\lambda_{2}-\lambda_{1},\\
\Omega_{yy}&=3\lambda_{1},\\
\Omega_{xy}&=0.\end{split}\end{equation*}
For this case the coefficients $A$ and $B$ are $A=\frac{3d-1}{2}$ and $B=\frac{27}{2}(1-d)d$, because $1/2\le d\leq1$ we see that these coefficients are non-negative for $\mu\in[0,1/2]$. The discriminant $D$ is given by the second order polynomial in the variable $d$, $D=\frac{225}{4}d^2-\frac{222}{4}d+\frac{1}{4}$. It is easy to verify that $D$ changes from negative to positive when $1/2\le d\leq1$. Because of the continuity of $D$ as a function of $d$, there exists $d_{0}$ such that $D(d_{0})=0$. We must recall that $d$ depends on $\mu$ also, therefore we have the following

\begin{prop} There exists a value $\mu_{0}$ such that $D=0$, as a consequence, the equilibrium points $L_{3}$ and $L_{4}$ have the following properties: for $\mu\in(0,\mu_{0})$ their eigenvalues are $\pm\textit{i}\omega_{1}$ and $\pm\textit{i}\omega_{2}$, for $\mu=\mu_{0}$ we have a pair of the eigenvalues $\pm\textit{i}\omega$ of multiplicity 2, finally when $\mu\in(\mu_{0},1/2]$ the eigenvalues are $\pm\alpha\pm\textit{i}\omega$ with $\alpha>0$ and $\omega>0$.
\end{prop}

By solving two quadratic equations, one for $D=0$ and other one for $d_{0}^2=1-3\mu+3\mu^2$ we obtain the value
$$
\mu_{0}=\frac{1}{224}\left(112-\sqrt{2(1979+37\sqrt{12097})}\right),
$$
which is approximately $\mu_{0}\approx 0.00898964$. In the  case of the solar system $\mu\in[0,0.00095]$ therefore the equilibrium points $L_{3}$ and $L_{4}$ are always linearly stable.

\section{Numerical explorations}

\subsection{Poincar\'e sections}
We first explore numerically the planar Hill problem $z=0$ by investigating the Poincar\'e sections and their dependence on the parameters  mass ratio $\mu$ and energy level. To do this, we first first regularize the collisions of the infinitesimal mass with $m_3$  via the Levi-Civita procedure.

\subsubsection{Regularization of the planar problem}
The Hamiltonian of the planar system \eqref{rotatedspatialhamiltonian} is
\[H(x,y,p_x,p_y)=\frac{1}{2}(p_x^2+p_y^2)+yp_x-xp_y+ax^2+by^2-\frac{1}{\sqrt{x^2+y^2}}.\]

In the sequel we consider the planar version of the problem ($z=0$) and
remove the singularity at the origin
using the Levi-Civita transformation.  The  Levi-Civita procedure  consists in changing the coordinates and  the conjugate momenta and in rescaling the time, as follows:
\[\left(\begin{array}{c} x \\ y \\ \end{array} \right)\longrightarrow\left(
\begin{array}{rr}  x & - y \\  y &  x \\\end{array} \right)
\left(\begin{array}{c}  x \\  y \\ \end{array} \right),\]
\[\left(\begin{array}{c} p_x \\ p_y \\ \end{array} \right)\longrightarrow\frac{2}{ x^2+ y^2}\left(
\begin{array}{rr}  x & - y \\  y &  x \\\end{array} \right)
\left(\begin{array}{c}  p_x \\  p_y \\ \end{array} \right),\]
and \[d\tau\longrightarrow\frac{4}{ x^2+ y^2}dt.\]

We recall that the  Levi-Civita transformation determines a double covering of the phase space.
The transformed Hamiltonian is \[\hat H(x,y)=\frac{( x^2+ y^2)^{1/2}}{4}(H(x,y)-h),\] where $h$ is the value of the non-regularized Hamiltonian $H$, and, with an abuse of notation, $(x,y,p_x,p_y)$ denote the transformed variables.

We obtain the following expression for $\hat H=\hat H(x,y,p_x,p_y)$:
\begin{equation}\begin{split}\hat H=&\left(-\frac{h}{2}\right)\left (\frac{x^2+y^2}{2} \right)+\frac{1}{2}(p_x^2+p_y^2)+ \left (\frac{x^2+y^2}{2} \right)(yp_x-xp_y)\\
&+\frac{1}{8}\left (ax^6+(4b-a)x^4y^2+(4b-a)x^2y^4+ay^6 \right) -\frac{1}{4} .
\end{split}\label{eqn:hillreg1}\end{equation}
We can omit from $\hat H$ the constant $-1/4$ (which implies that an $h$-level set  of the Hamiltonian  $H$ corresponds to a  $(1/4)$-level set of the Hamiltonian $\hat H$).

Since $\hat H$ depends on the value $h$ of the non-regularized Hamiltonian, we eliminate it through a canonical change of variables
\begin{equation}\label{gamma} x=\alpha X,\, y=\alpha Y,\, p_x=\beta P_x,\, p_y=\beta P_y,\, \hat H=\gamma\check H\end{equation}
where $\alpha =2(-h/2)^{1/4}$, $\beta=2(-h/2)^{3/4}$ and $\gamma=4(-h/2)^{3/2}$. This transformation is valid when  $h<0$; when $h>0$ we can use the same scaling with $h$ instead of $-h$. Thus, both $h$ and $-h$ correspond to the same value $\check h$ of the new Hamiltonian $\check H$.
When $h\to 0$ the value of the  Hamiltonian $\check H$ approaches $+\infty$.

The new, regularized  Hamiltonian $\check H=\check H(X,Y,P_X,P_Y)$ is given by
\begin{equation}\begin{split}\check H=&\frac{1}{2}(X^2+Y^2+P_X^2+P_Y^2)+2(X^2+Y^2)(YP_X-XP_Y)\\
&+2\left ( aX^6+(4b-a)X^4Y^2+(4b-a)X^2Y^4+aY^6\right).
\end{split}\label{eqn:hillreg2}\end{equation}
From \eqref{gamma} we have that the  Jacobi constant $C_h=-2h$ is related to the energy level  $\check h$ of the regularized Hamiltonian $\check H$ by
\begin{equation}\label{energyvsjac}\check h =(|C_h|^{-3/2})/2.\end{equation}

Note that in the special case $\mu=0$ we have $\lambda_1=0$, $\lambda_2=3$, $a=-2$, $b=1$, and the regularized Hamiltonian is
\begin{equation}\begin{split}\check H=&\frac{1}{2}(X^2+Y^2+P_X^2+P_Y^2)+2(X^2+Y^2)(YP_X-XP_Y)\\
&-4\left ( X^6-3X^4Y^2-3X^2Y^4+Y^6\right),
\end{split}\label{eqn:hillreg2}\end{equation}
which is the same as for  classical Hill lunar problem  in Levi-Civita regularized coordinates \cite{Simo}.

The corresponding Hamilton equations  to \eqref{eqn:hillreg2} are
\begin{equation}\begin{split}
\dot X&=P_X+2(X^2+Y^2)Y,\\
\dot Y&=P_Y-2(X^2+Y^2)X,\\
\dot P_X&=-X+2(X^2+Y^2)P_Y-4X(YP_X-XP_Y)\\&-\left [12aX^5+8(4b-a)X^3Y^2+4(4b-a)XY^4\right],\\
\dot P_Y&=-Y-2(X^2+Y^2)P_X-4Y(YP_X-XP_Y)\\&-\left[4(4b-a)X^4Y+8(4b-a)X^2Y^3+12aY^5\right].
\end{split}\label{eqn:hillreg2}\end{equation}

We remark   that $\check H$ depends on the mass parameter $\mu$.  The first two terms of $\check H$, consisting of a homogeneous polynomial of degree $2$ and a
homogeneous polynomial of degree $4$ in $(X,Y,P_X,P_Y)$, correspond to a Hamiltonian system describing the motion of two uncoupled oscillators perturbed by a Coriolis force. It is Liouville-Arnold integrable. The third term,
consisting in a homogeneous polynomial of degree $6$, represents the perturbation  due to the two main bodies (e.g., Sun and Jupiter). The numerical experiments, shown  below, suggest that this term makes the system non-integrable, as they reveal the well known KAM phenomena. The non-integrability of the Hill problem in the lunar case, corresponding to $\mu=0$,  has been proved in \cite{Meletlidou,Winterberg,MoralesRuiz}.

\subsubsection{Poincar\'e sections for various mass ratios and energy levels}

We explore numerically the dynamics of the regularized Hamiltonian $\check H$ given by \eqref{eqn:hillreg2}.
We compute the first return map to the Poincar\'e section given  by $Y=0$, $P_Y>0$, for various choices of mass ratios $\mu$ and Jacobi constants $C$, which is related to the energy level $\check h$ of $\check H$ by \eqref{energyvsjac}.

In Fig. \ref{hill.eps} we show  the Poincar\'e sections for the Hill restricted four-body problem  with $\mu=0$, at the energy level  $C=4.329636$, which corresponds to the `classical' lunar Hill problem \cite{Simo}.
This is the energy level at which a stochastic `maple leaf' shaped region appears.

\begin{figure}
\centering
\includegraphics[width=0.4\textwidth]{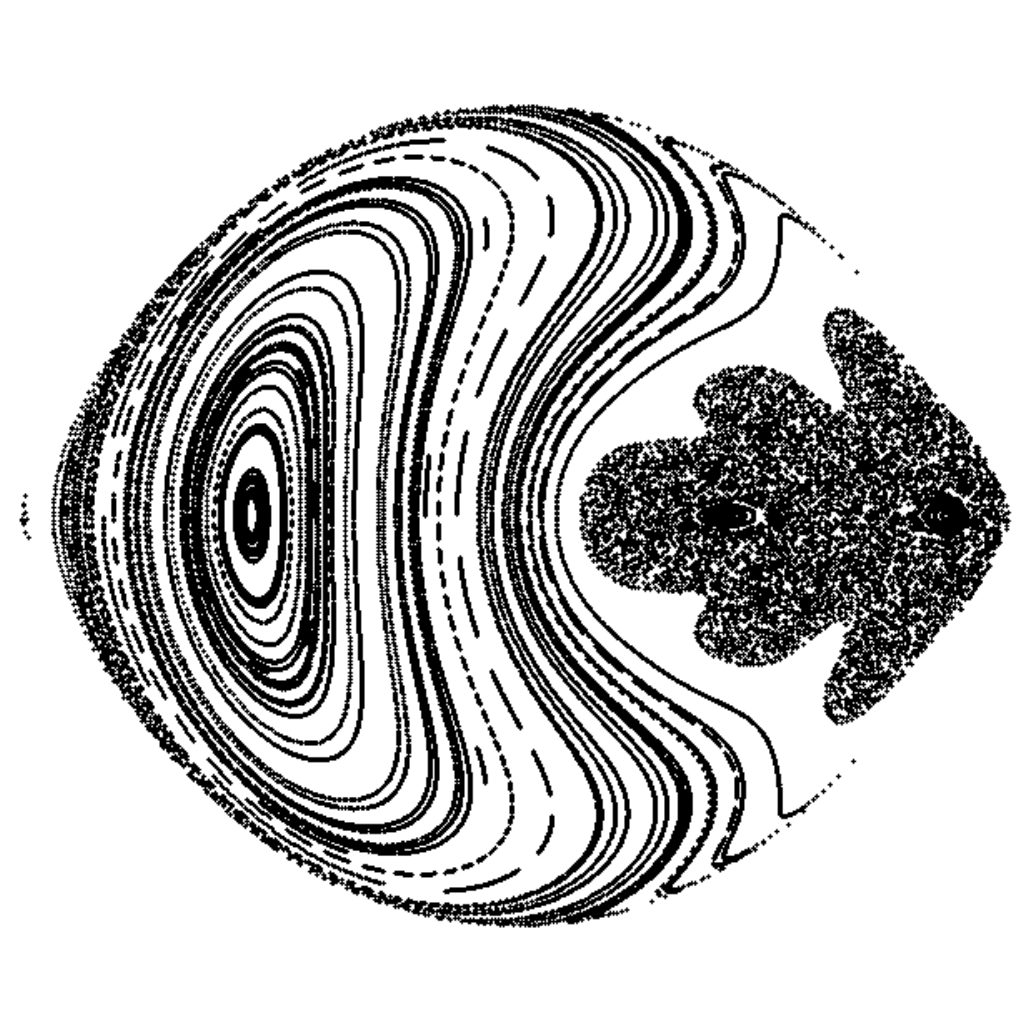}
\caption{Poincar\'e sections for the Hill restricted four-body problem $C = 4.329636$ and $\mu=0$.}
\label{hill.eps}
\end{figure}

In Fig. \ref{hill_mu01} we show the Poincar\'e sections for  the Hill restricted four-body problem with $\mu=0.1$ for various values of the  Jacobi constant  $C$. For small   there are two stable fixed points, one on the left corresponding to retrograde motion, and  one on the right corresponding to direct motion. As $C$ is decreased  the fixed point on the right undergoes a pitchfork bifurcation, so is becoming unstable and  two other stable fixed points appear. After this, the region on the right becomes more and more chaotic and a `maple leaf'  region surrounding elliptic islands appears. When $C$ is decreased even further, the zero velocity curve bounding the  region on the right opens up and direct orbits  escape to the exterior region.

This type of pitchfork bifurcation occurs for  all mass ratios. As an example, in figure (\ref{bifurcation}) we show the pitchfork bifurcation of the family of direct periodic orbits around the tertiary for the mass parameter $\mu=0.00095$ on the plane $(C,x_{0})$, where $x_{0}$ stands for the positive intersection of the orbit with the $x-$axis. The bifurcation occurs approximately for the periodic orbit corresponding to $(C,x_{0})=(4.4983599991, 0.2836529981)$. This family has been referred to as the $g$-family in the classical works of M. H\'enon \cite{Henon}. The periodic orbits in the figure are shown in the physical coordinates. In a forthcoming work we will provide more details on the structure of the families of planar periodic orbits of this system.

\begin{figure}
  \centering
  \includegraphics[width=2.2in]{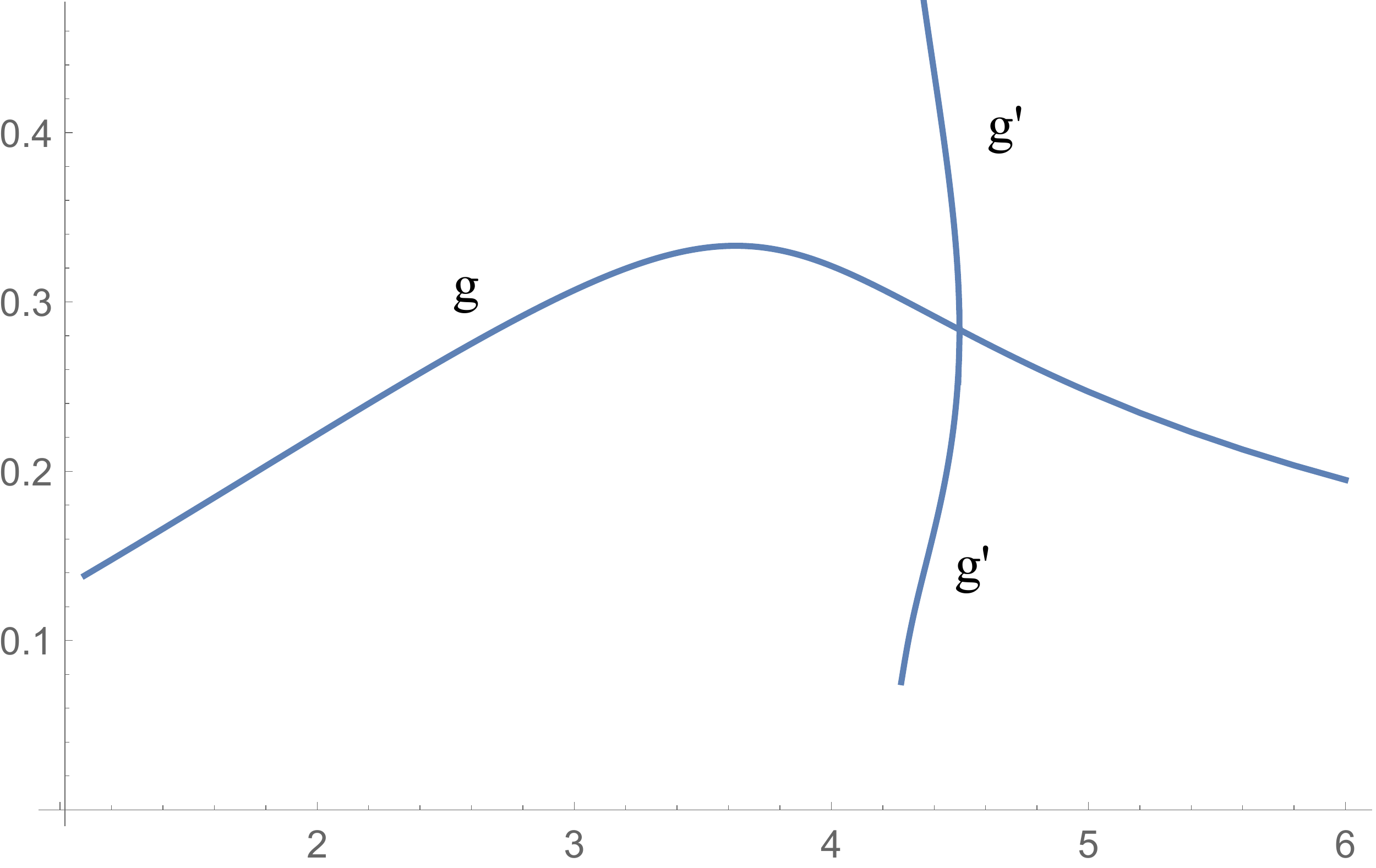}\quad
  \includegraphics[width=2.4in]{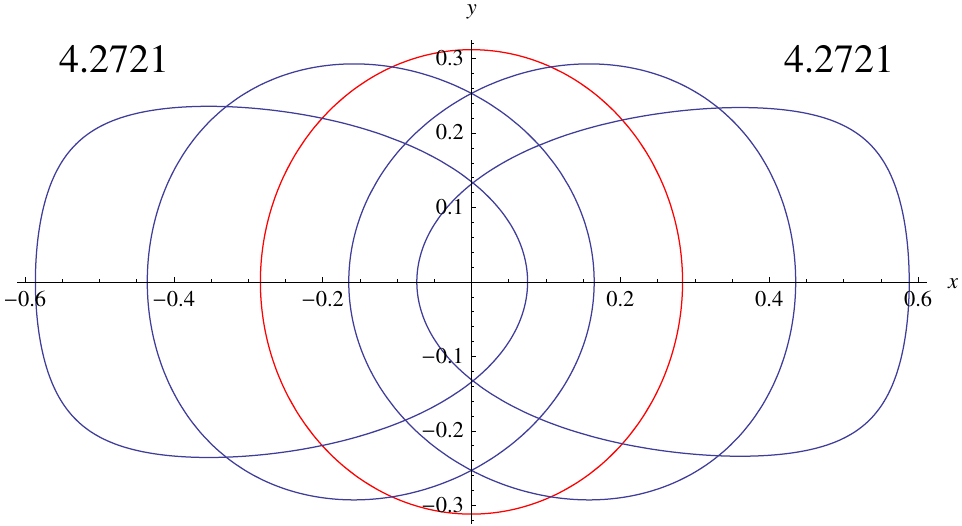}\\
 \caption{Left. Characteristic curves for the $g$-family  and the bifurcating branches $g'$ in the plane $(C,x_{0})$. Right. Evolution of the periodic orbits after the bifurcating periodic orbit (in red). The periodic orbits on the right corresponds to the upper branch $g'$ and the periodic orbits on the left corresponds to the lower branch $g'$. }\label{bifurcation}
\end{figure}

In Fig. \ref{hill_mu05} we show the Poincar\'e sections for  the Hill restricted four-body problem with $\mu=0.5$ for various values of the Jacobi constant $C$.  The behavior is similar, with the major difference that the pitchfork bifurcation   occurs after the zero velocity curve opens.

\begin{figure}$\begin{array}{cc}
\includegraphics[width=0.4\textwidth]{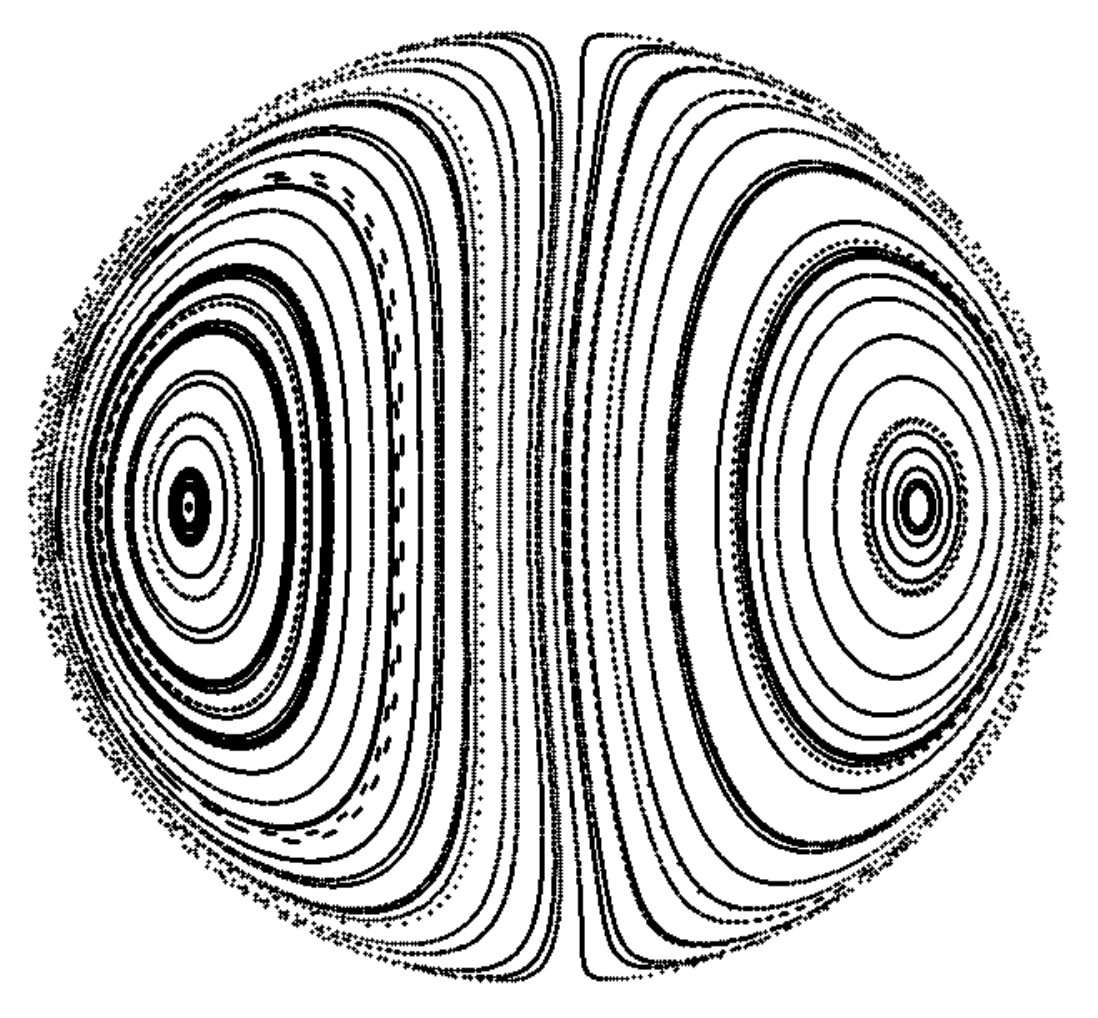}&
\includegraphics[width=0.4\textwidth]{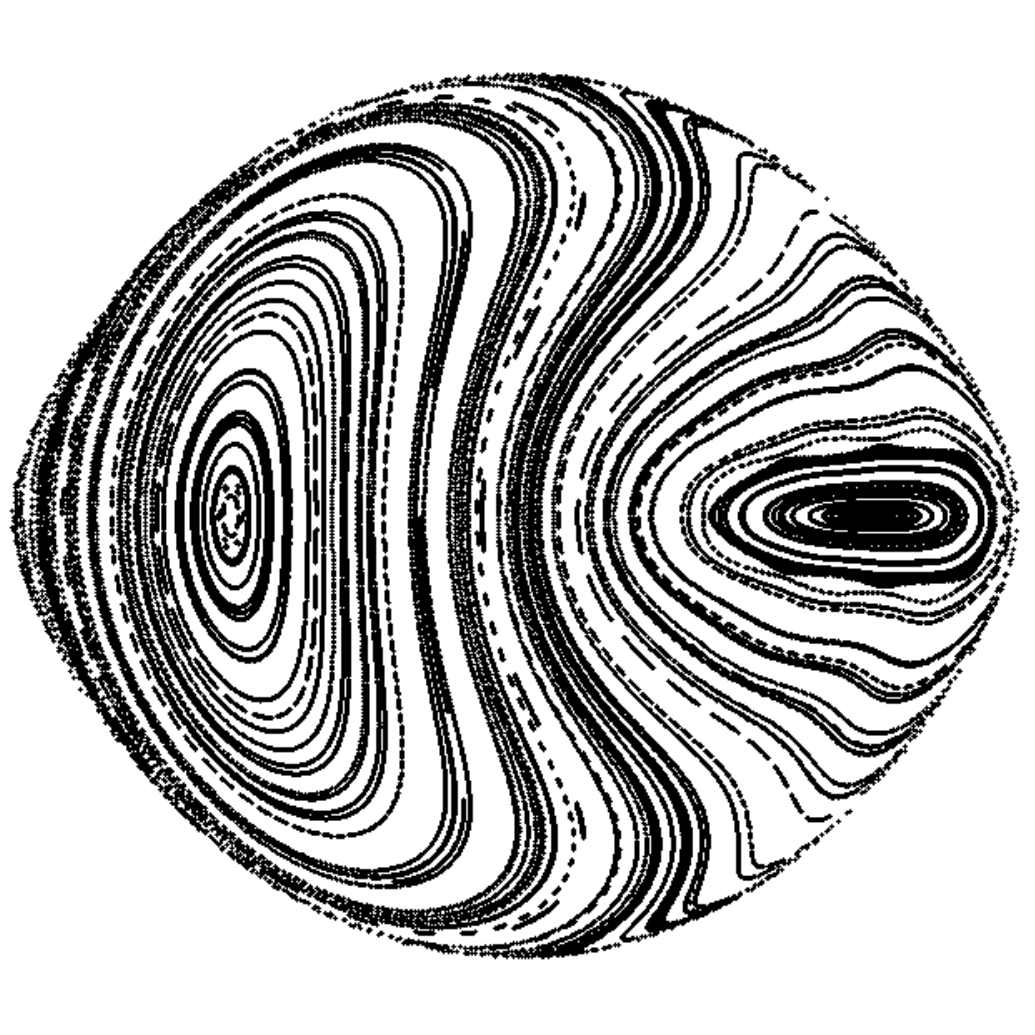}\\
\includegraphics[width=0.4\textwidth]{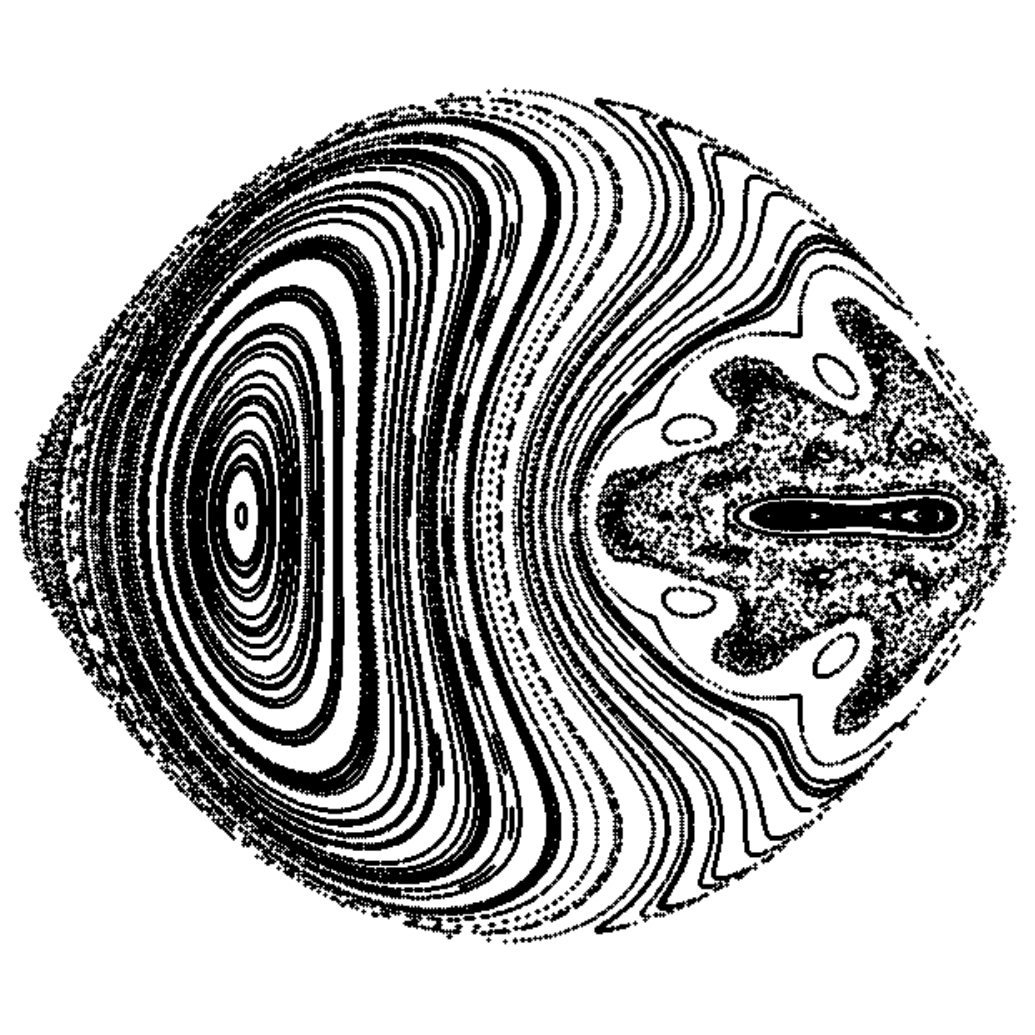}&
\includegraphics[width=0.4\textwidth]{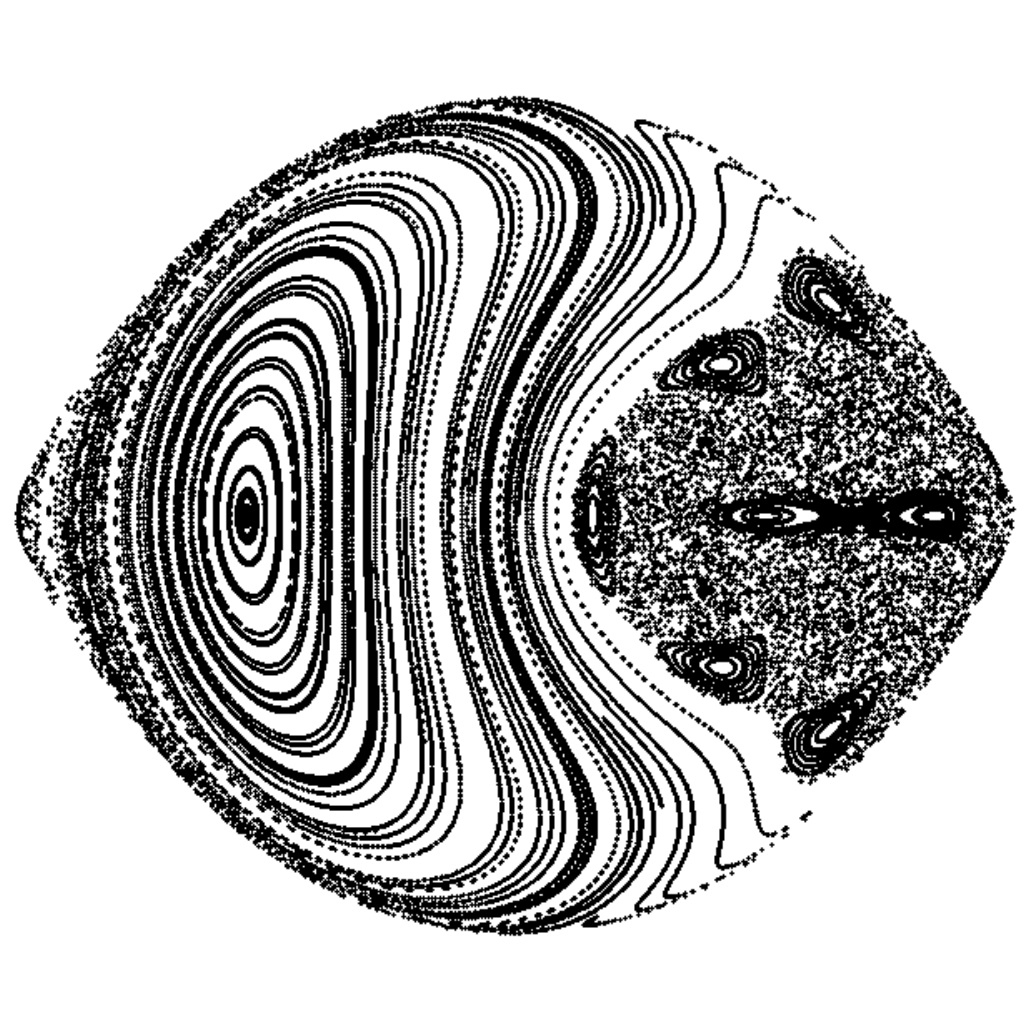}\\
\includegraphics[width=0.4\textwidth]{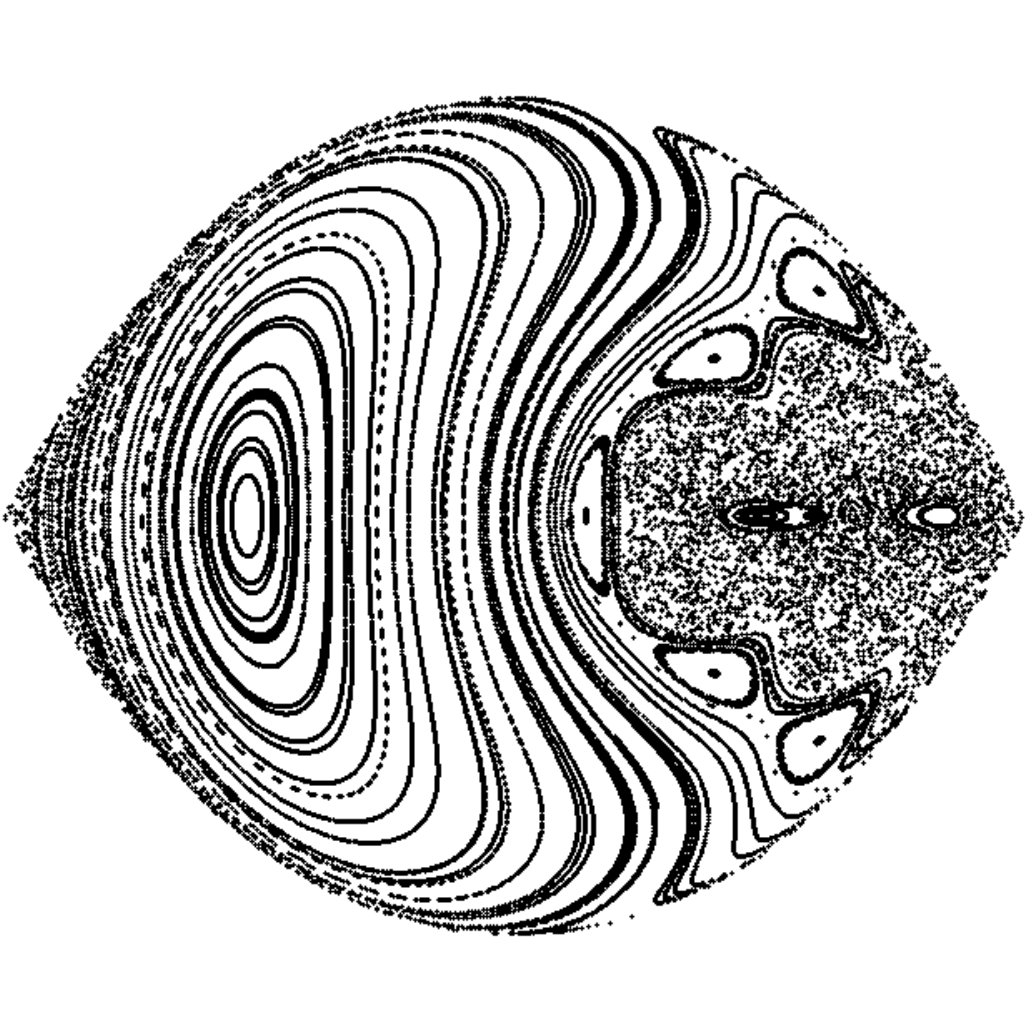}&
\includegraphics[width=0.4\textwidth]{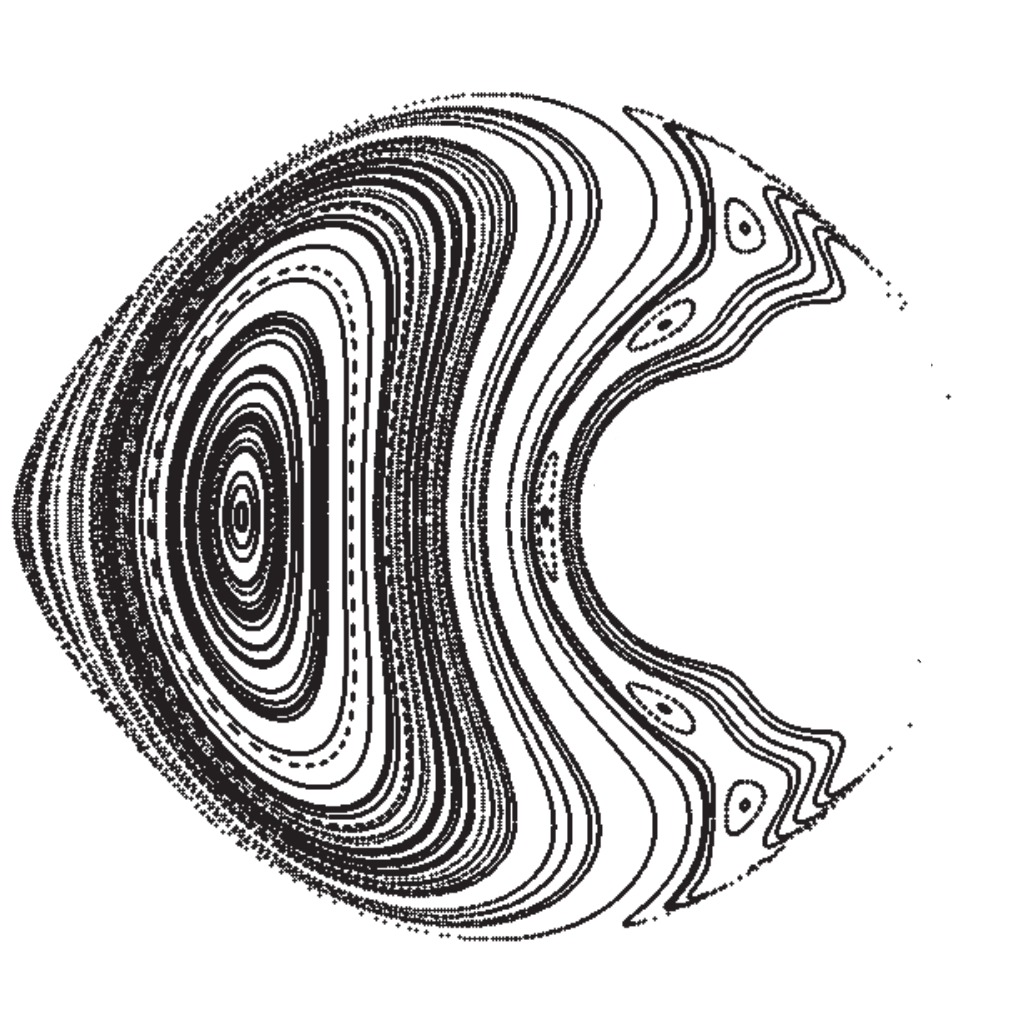}\\
\end{array}$
\caption{Poincar\'e sections for the Hill restricted four-body problem
(a)   $C=13.57209$ and $\mu=0.1$;
(b)   $C=4.329636$ and $\mu=0.1$;
(c)   $C=4.25334$ and $\mu=0.1$;
(d)   $C= 4.228647$ and $\mu=0.1$;
(e)   $C= 4.21887$ and $\mu=0.1$;
(f)   $C= 4.110353$ and $\mu=0.1$.
}
\label{hill_mu01}
\end{figure}

\begin{figure}$\begin{array}{cc}
\includegraphics[width=0.4 \textwidth]{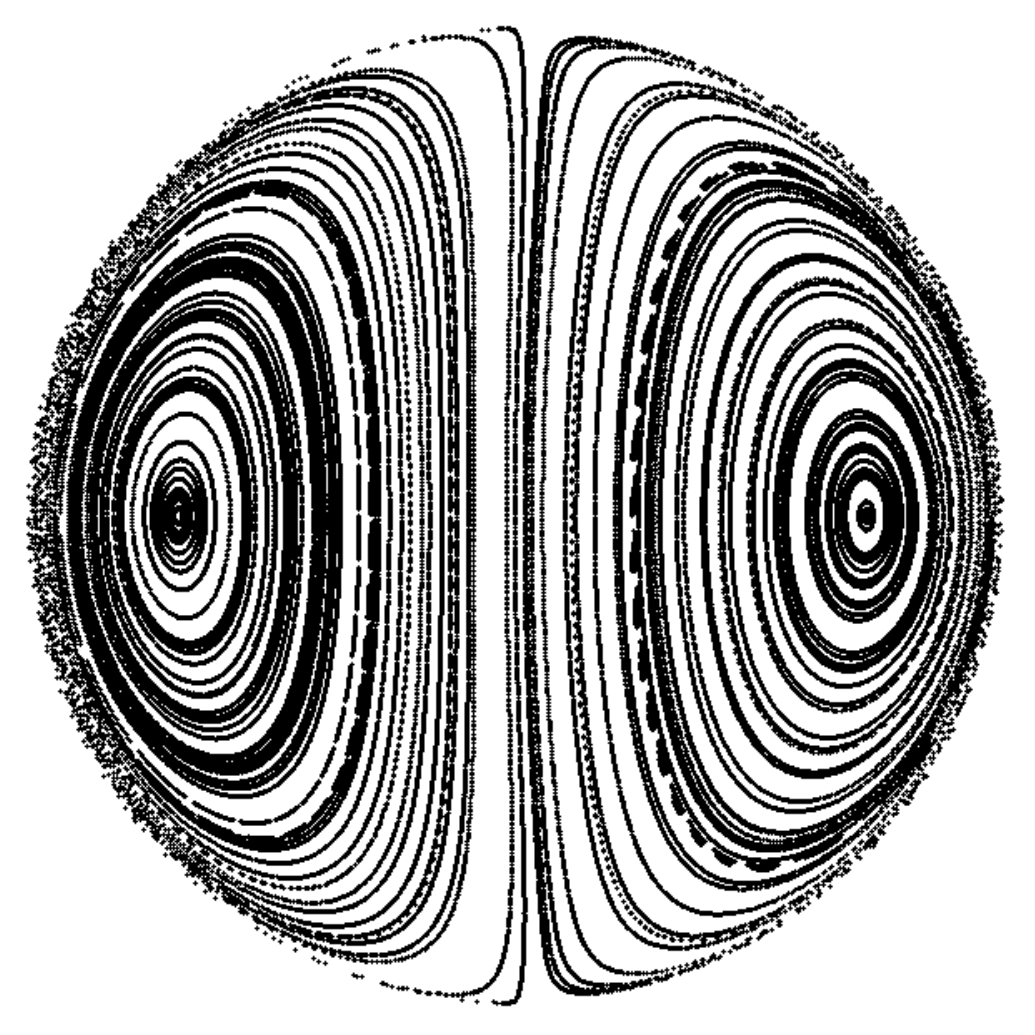}&
\includegraphics[width=0.4 \textwidth]{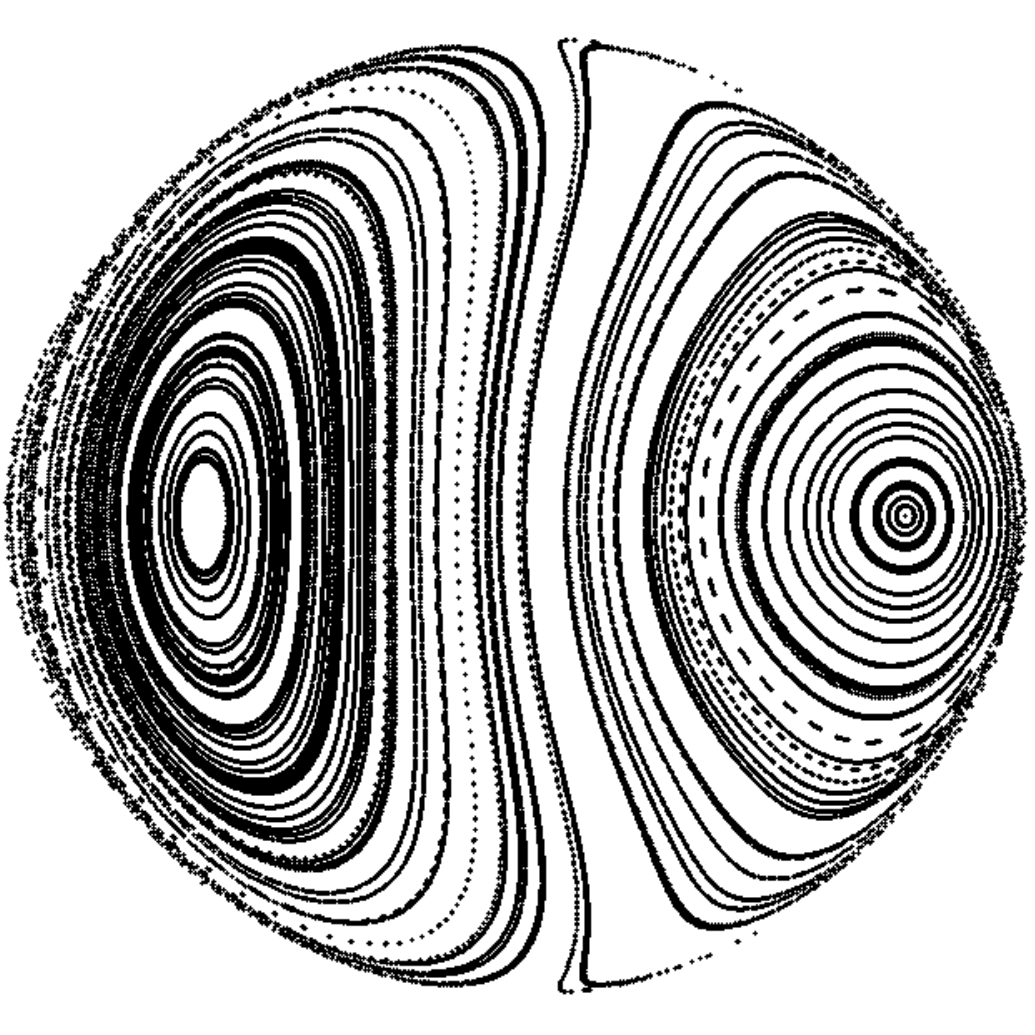}\\
\includegraphics[width=0.4 \textwidth]{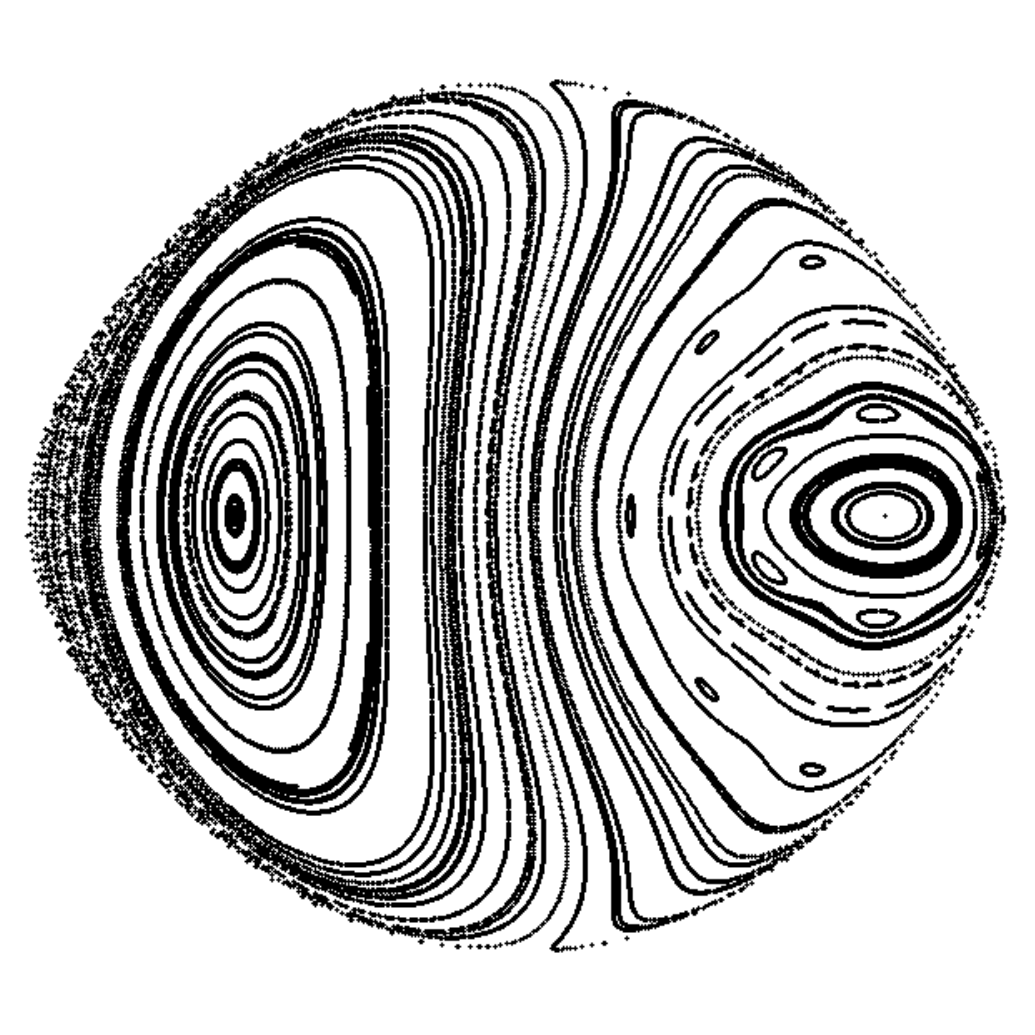}&
\includegraphics[width=0.4 \textwidth]{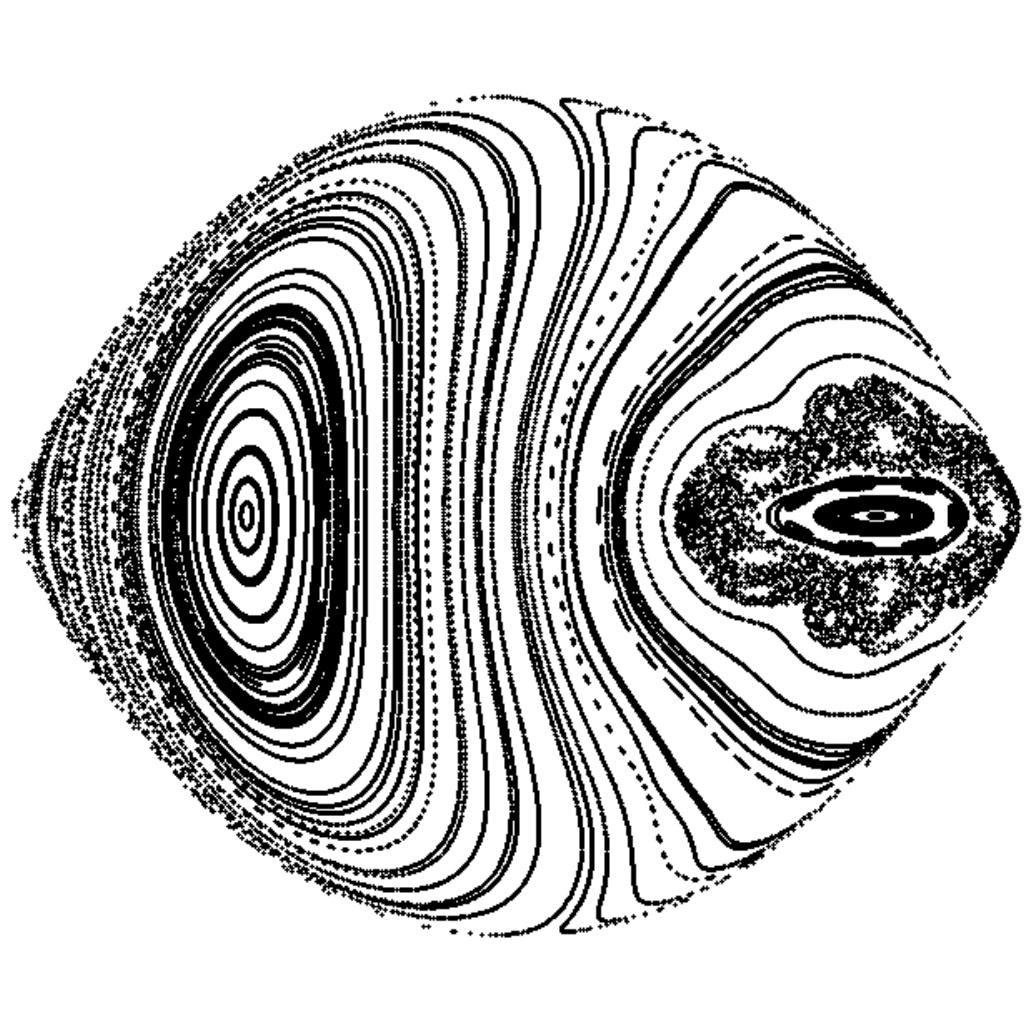}\\
\includegraphics[width=0.4 \textwidth]{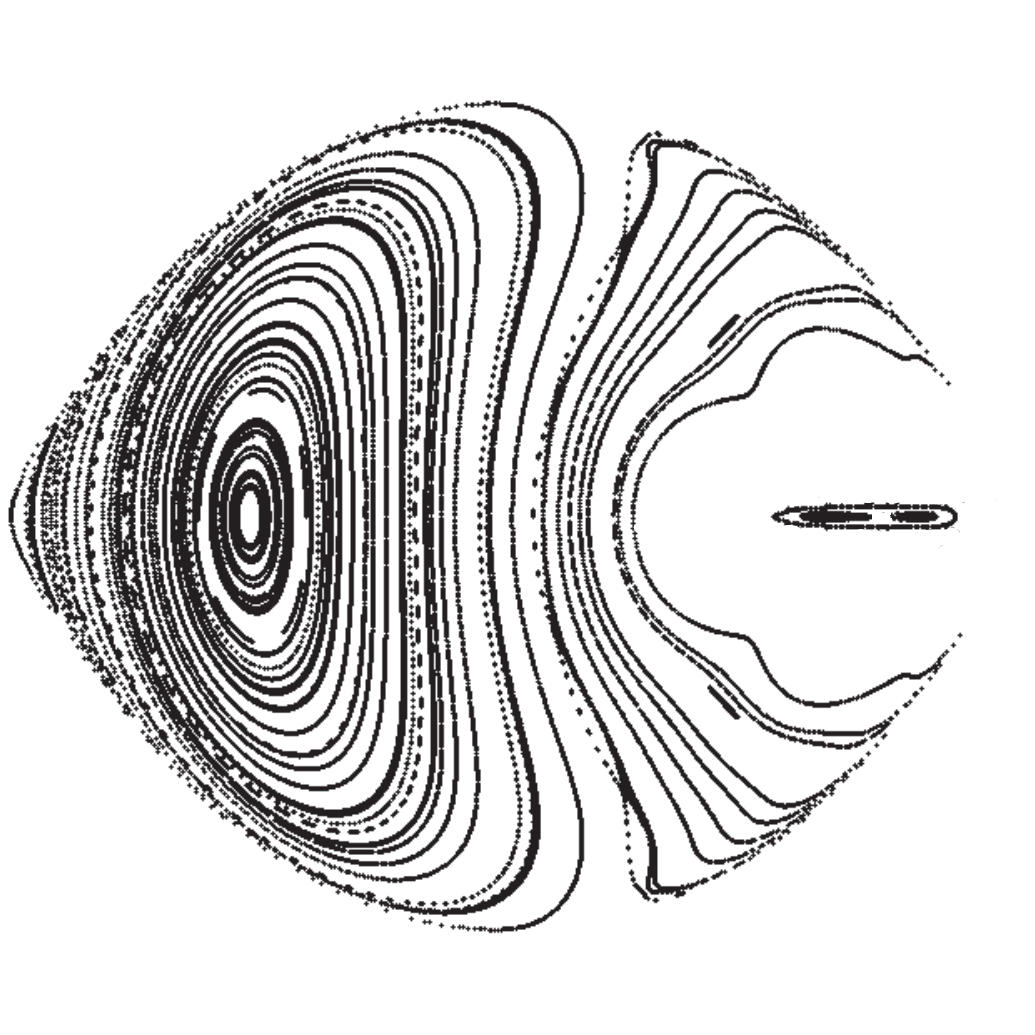}&
\includegraphics[width=0.4 \textwidth]{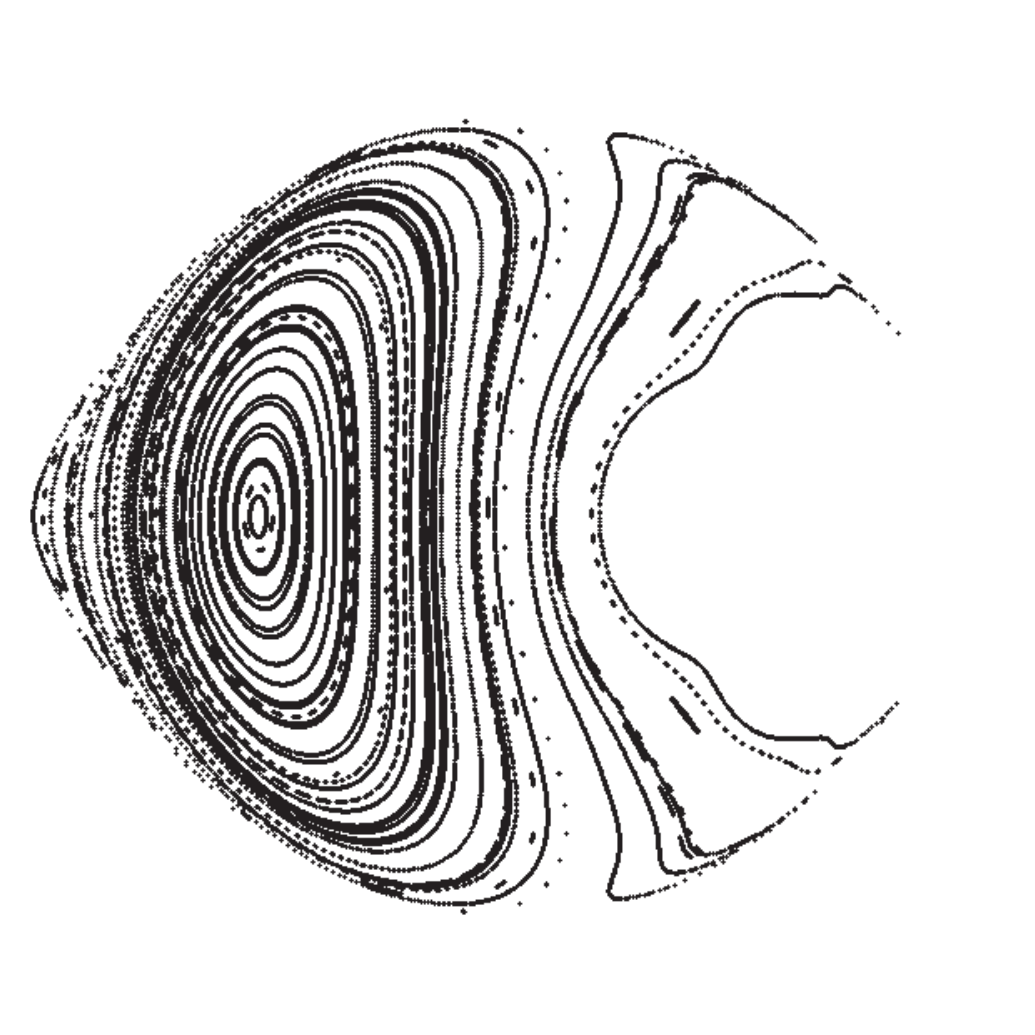}\\
\end{array}$
\caption{Poincar\'e sections for the Hill restricted four-body problem
(a)   $C= 13.57209$ and $\mu=0.5$;
(b)   $C= 4.641589$ and $\mu=0.5$;
(c)   $C=  4.110353$ and $\mu=0.5$;
(d)  $C=3.937253$ and $\mu=0.5$;
(e) $C=  3.882841$ and $\mu=0.5$;
(f) $C=3.818828$ and $\mu=0.5$.
}
\label{hill_mu05}
\end{figure}

\subsection{Numerical explorations of the invariant manifolds of the Lyapunov orbits of the equilibrium point $L_{1}$.}

In this section we perform a numerical exploration of the stable ($W^{s}$) and unstable ($W^{u}$) manifolds of the Lyapunov orbits for the saddle-center equilibrium points $L_{1}$ and $L_{2}$. Because of the symmetries of the equations of motion, it will be enough to study the invariant manifolds for the equilibrium point $L_{1}$, the corresponding manifolds for $L_{2}$ can be obtained by symmetry. The numerical explorations were performed using Hill's equations for the restricted four body problem with the mass parameter $\mu=0.00095$ that corresponds to the mass ratio of Jupiter-Sun, The value of the Jacobi constant at this equilibrium point is $C_{L_{1}}=4.32572$. In the Fig. \ref{liapunovorbits} we can show the evolution of the family of the periodic orbits emanating of $L_{1}$  for several values of the Jacobi constant.
\begin{figure}
\centering
\includegraphics[width=2.0in]{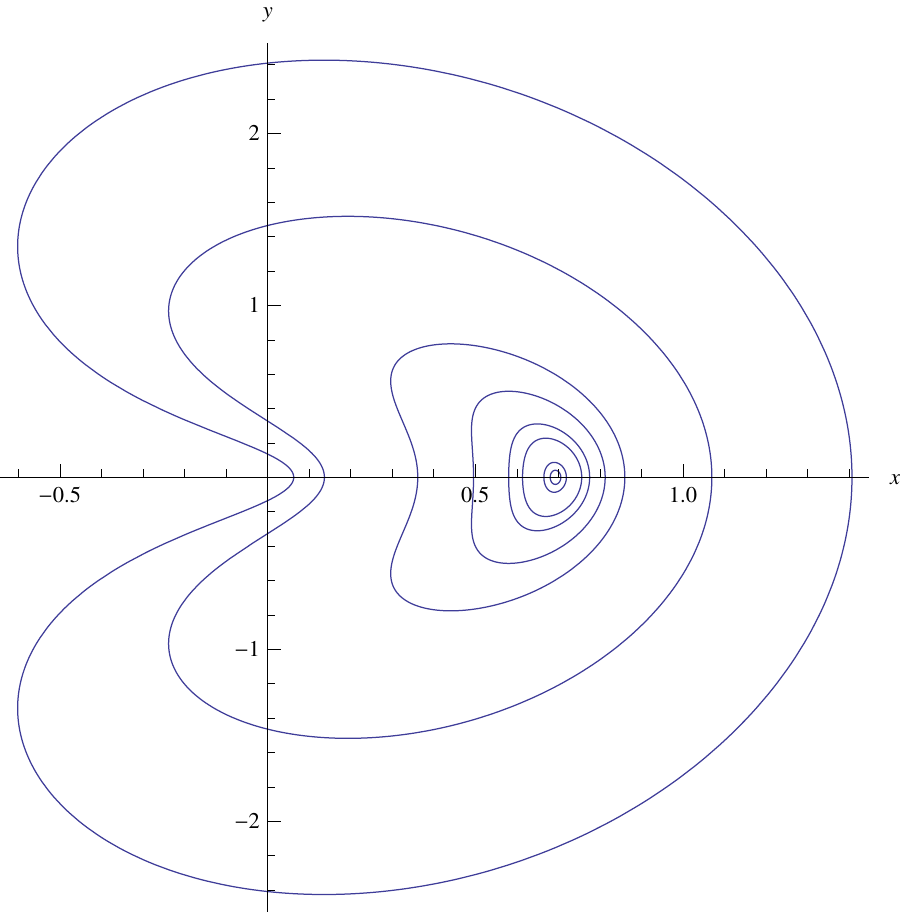}
\caption{Evolution of the family of Lyapunov orbits.\label{liapunovorbits}}
\end{figure}

\subsection{The inner region}

The inner region corresponds to the small (blue) region around the tertiary as it is shown in the Fig. \ref{limithillregions}. In order to visualize the behavior of the invariant manifolds in this region, we choose the Poincar\'e section $\Sigma:=\{(x,y,\dot{x},\dot{y})\in\mathbb{R}^{4}\vert x=0\}$ that corresponds to the intersections of trajectories with the $y-$axis. This section contains two subsections $\Sigma^{+}:=\{(x,y,\dot{x},\dot{y})\in\mathbb{R}^{4}\,\vert\, x=0,y>0\}$ and $\Sigma^{-}:=\{(x,y,\dot{x},\dot{y})\in\mathbb{R}^{4}\,\vert\, x=0,y<0\}$, and we are going to consider  crossings (cuts) made by manifolds with these subsection. A trajectory  that intersects transversely the section $\Sigma$, i.e., with velocity $\dot{x}\ne0$, is uniquely determined by the coordinates $(y,\dot{y})$ of the intersection point, as we can obtain the initial condition for the trajectory by considering $x=0$ and solving for $\dot{x}$ from the first integral. A trajectory is tangential to $\Sigma$ if  it intersects the surface section with $\dot{x}=0$, therefore the coordinates $(y,\dot{y})$ at the tangency points can be obtained from the first integral  as $$\dot{y}^{2}=2\Omega(0,y,\mu)-C,$$  which defines a curve in the plane $(y,\dot{y})$ that depends on the mass parameter $\mu$ and the Jacobi constant $C$.

First we choose the value $C=4.3$ close to $C_{L_{1}}$ for the computation of the invariant manifolds $W^{u}$ and $W^{s}$ of the Lyapunov orbits.
The first few cuts  $W^{u}$ and $W^{s}$, viewed as subsets of $(\dot{y},y)$-are diffeomorphic to circles. In Fig.  \ref{5cuts} we show the first five intersections of the invariant manifolds with $\Sigma$ in the so called inner region. The unstable manifold is shown in blue and the stable one is shown in red. In the following we adopt the notation $W_{n}^{s}$ and $W_{m}^{s}$ where $n$ and $m$ count the number of cuts with the surface section. We stress that  this notation takes  into account the cuts with either  subsection $\Sigma^{+}$ or  subsection $\Sigma^{-}$.

\begin{figure}
  \centering
  \includegraphics[width=1.4in]{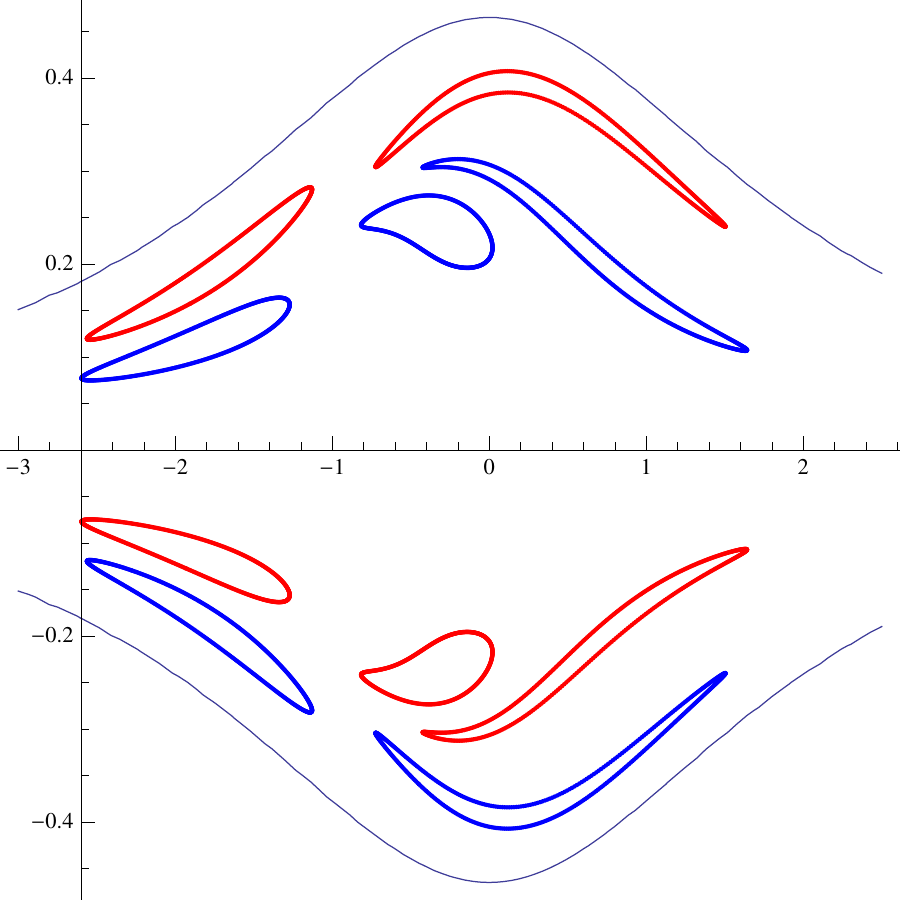}\quad
  \includegraphics[width=1.4in]{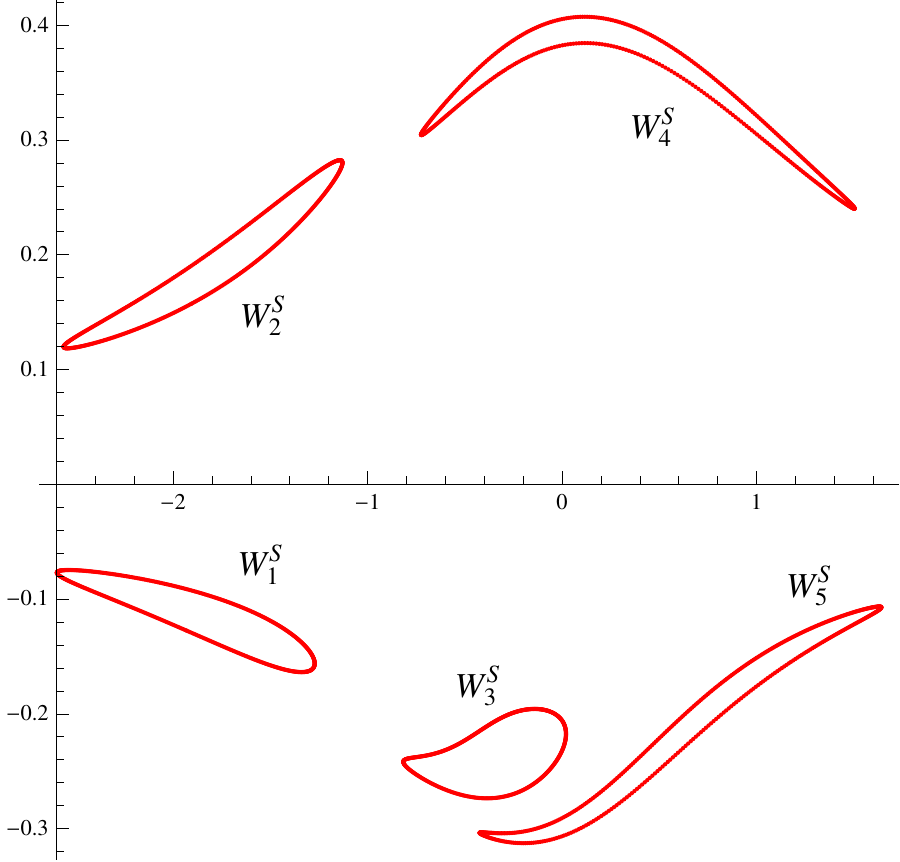}\\
  \includegraphics[width=1.4in]{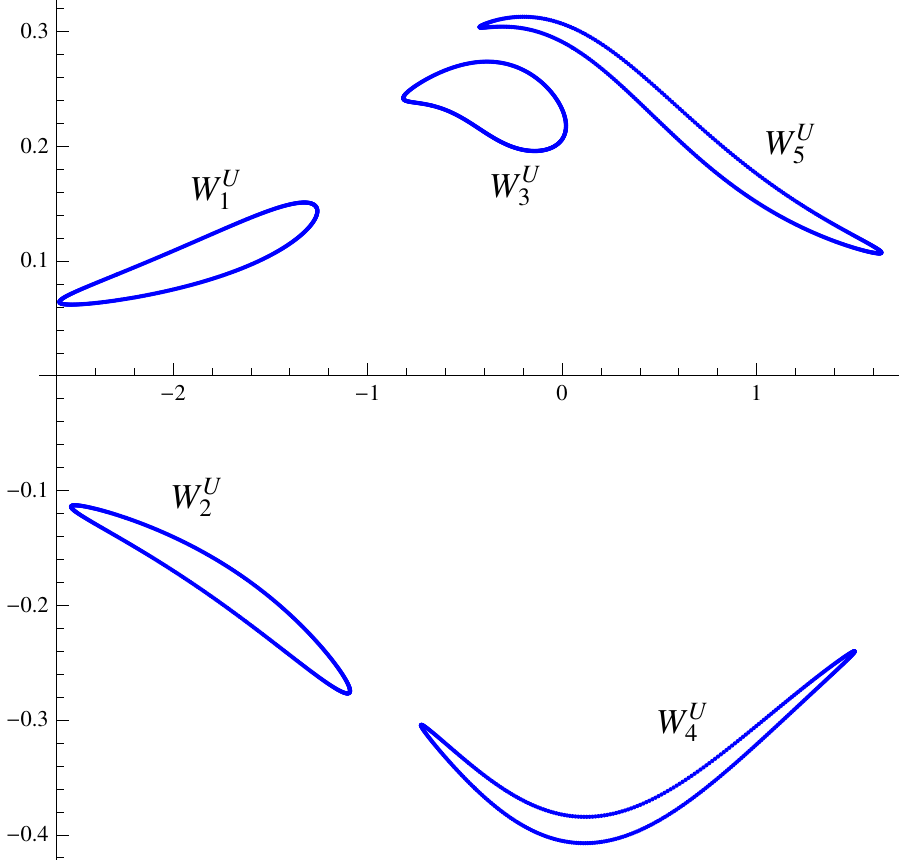}
      \caption{Stable and unstable manifolds after five cuts with the Poincar\'e section $\Sigma$ and tangency curve in the plane $(\dot{y},y)$, top right. The five cuts of $W^{s}$ are shown in top left, and the five cuts of $W^{u}$ are shown in second row.}\label{5cuts}
\end{figure}
\begin{figure}
\centering
\includegraphics[width=2.0in]{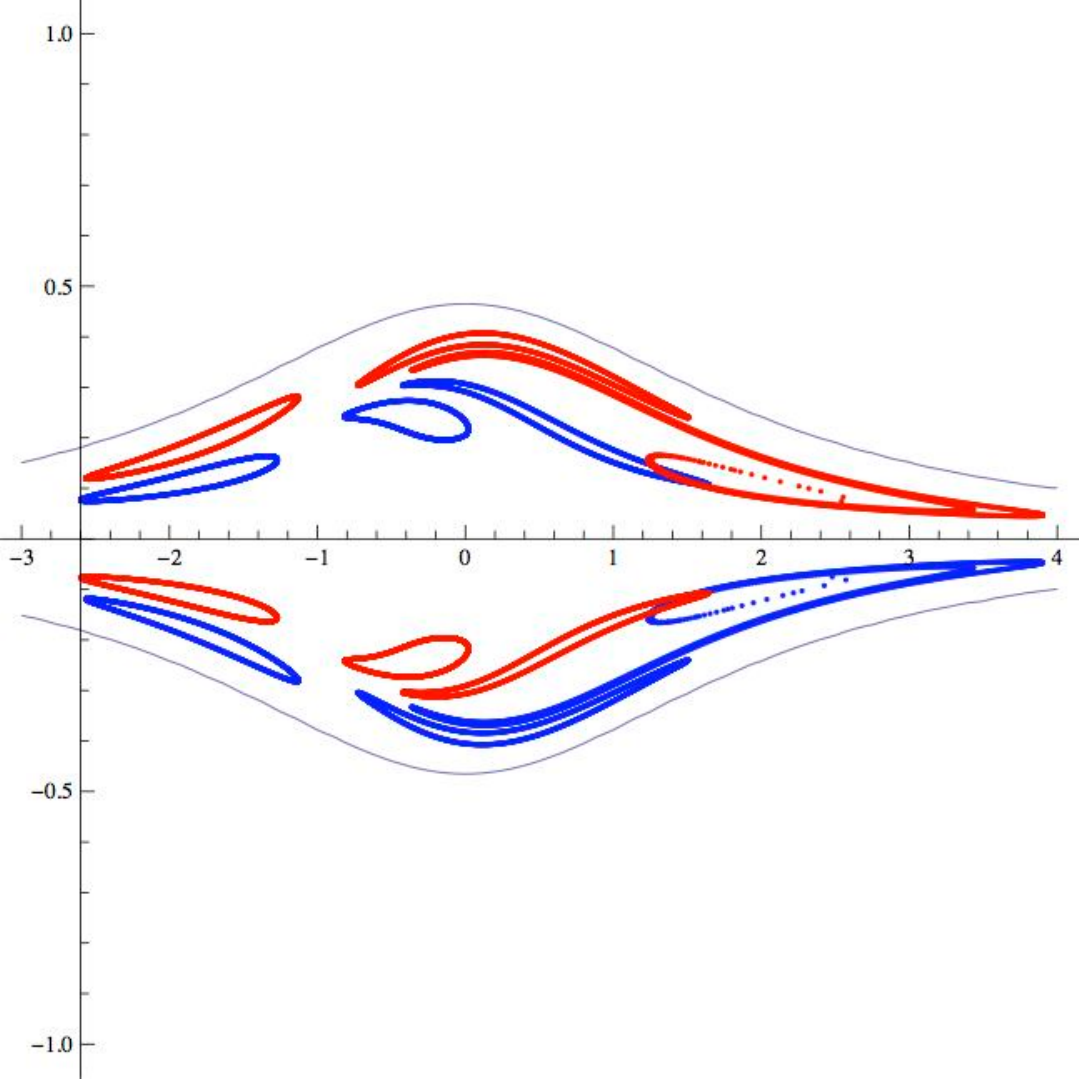}
\caption{The stable and unstable manifolds after six cuts with the Poincar\'e section in the plane $(\dot{y},y)$.\label{sixcuts}}
\end{figure}

At the sixth cut with the surface of section, shown in Fig. \ref{sixcuts}, we detect the first intersections between the invariant manifolds, as the intersections  between $W_{5}^{u}$ and $W_{6}^{s}$, and also, because of the symmetry of the equations,  as the intersections  between $W_{6}^{u}$ and $W_{5}^{s}$; see Fig. \ref{firstintersections}. After this first intersections the cuts determined by the invariant manifolds on the $(\dot{y},y)$-plane  are no longer diffeomorphic to circles, moreover, as expected, a transverse homoclinic point begets other homoclinic points near the original one. The intersection between $W_{5}^{u}$ and $W_{6}^{s}$ occurs in the subsection $\Sigma^{+}$; following  $W^{s}$  we find that the next cut with $\Sigma^{+}$, denoted by $W_{8}^{s}$, intersects with $W_{5}^{u}$; see Fig. \ref{secondintersections} (a). Analogously, the intersection between $W_{6}^{u}$ and $W_{5}^{s}$ occurs in the subsection $\Sigma^{-}$; following   $W^{s}$ we find that the next cut with $\Sigma^{-}$, denoted by $W_{7}^{s}$, intersects with $W_{6}^{u}$; see Fig. \ref{secondintersections} (b). In a similar way we can consider the next cuts of $W^{u}$ with $\Sigma^{-}$ and $\Sigma^{+}$ to find transverse intersections between $W_{8}^{u}$ and $W_{5}^{s}$, and between $W_{7}^{u}$ and $W_{5}^{7}$, respectively.

In the Fig. \ref{tencuts} we show the transverse intersections for some of the subsequent cuts of the manifolds with the section $\Sigma$.

\begin{figure}
  \centering
  \includegraphics[width=1.4in]{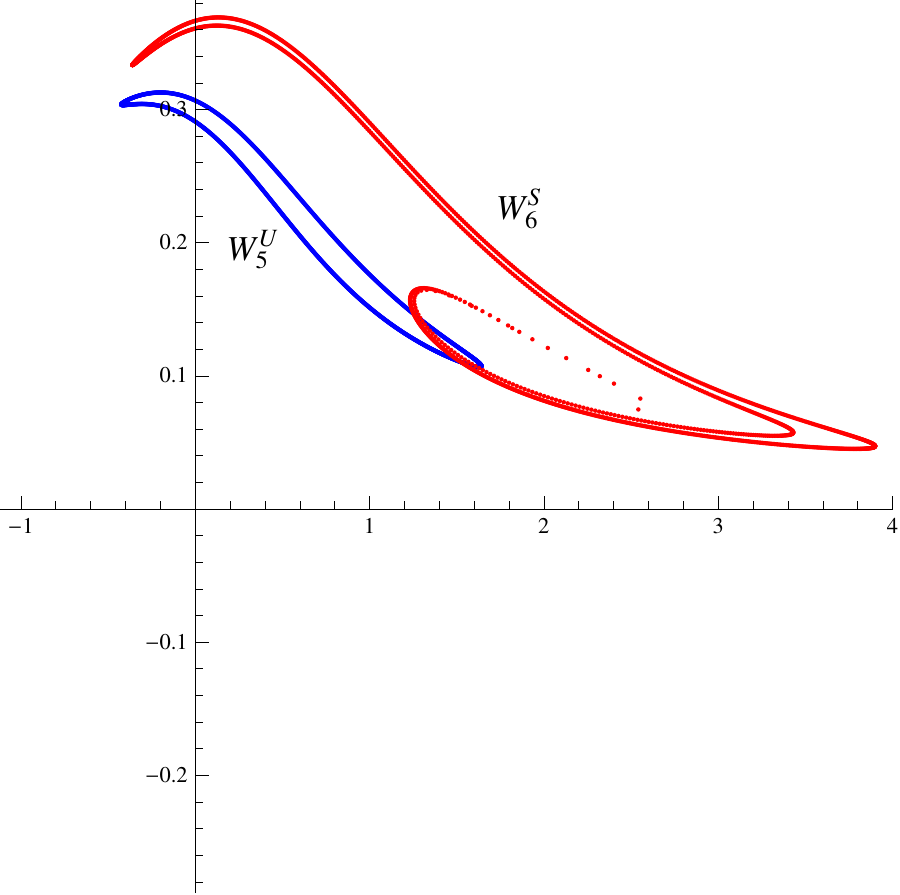}\quad
  \includegraphics[width=1.4in]{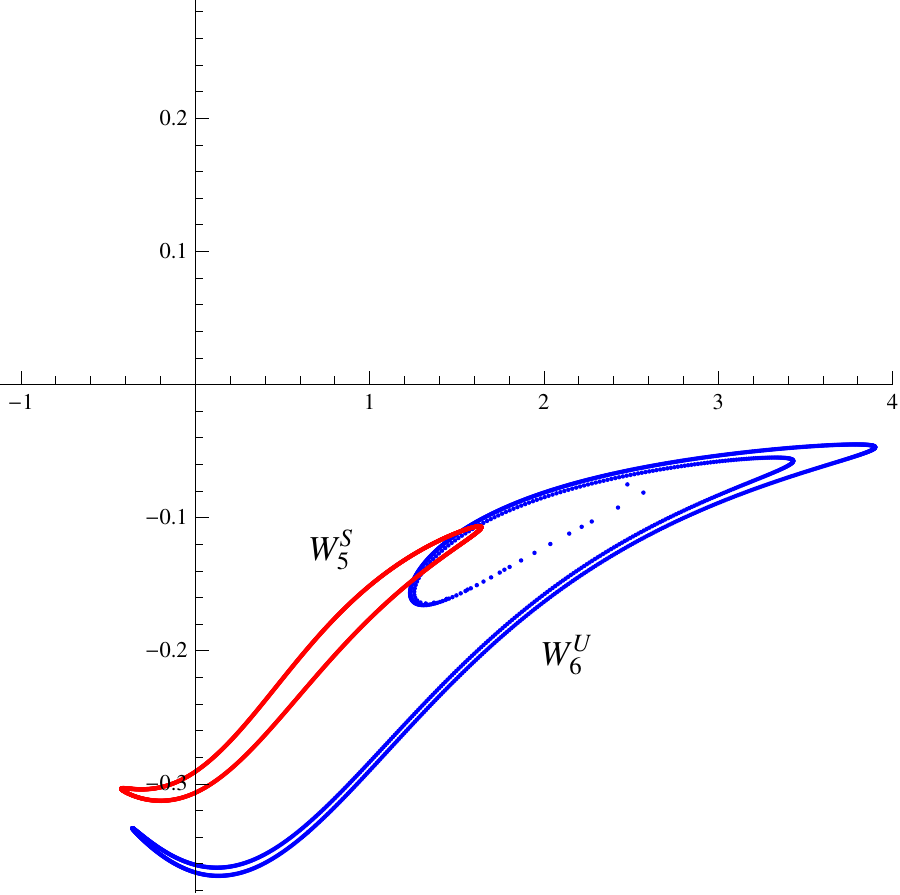}\\
 \caption{Intersection between $W_{5}^{u}$ and $W_{6}^{s}$ in the plane $(\dot{y},y)$, left. Intersection between $W_{6}^{u}$ and $W_{5}^{s}$, right.}\label{firstintersections}
\end{figure}

\begin{figure}
  \centering
  \includegraphics[width=1.4in]{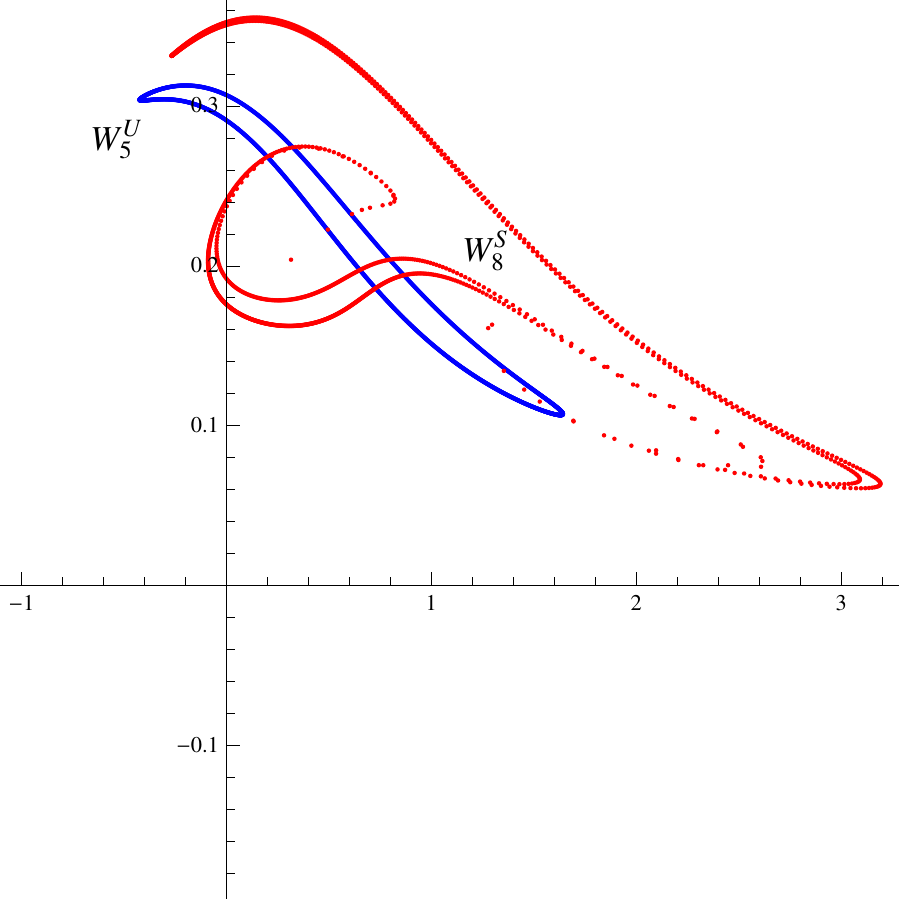}\quad
  \includegraphics[width=1.4in]{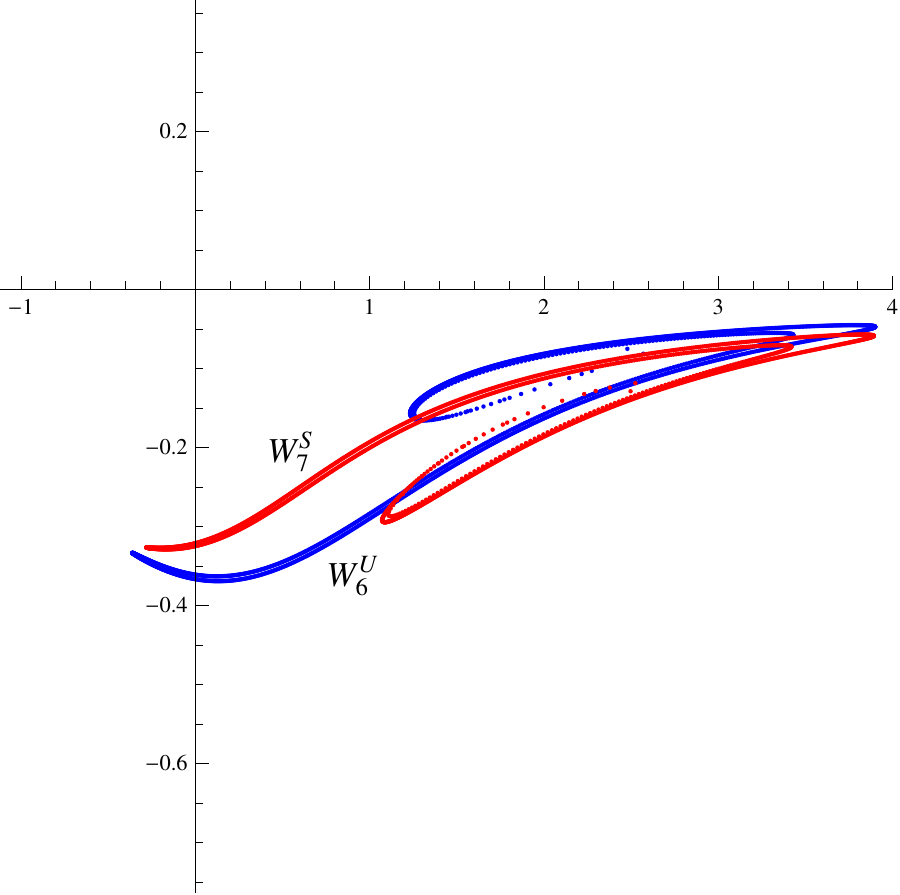}\\
 \caption{Intersection between $W_{5}^{u}$ and $W_{8}^{s}$ in the plane $(\dot{y},y)$, left. Intersection between $W_{6}^{u}$ and $W_{7}^{s}$, right.}\label{secondintersections}
\end{figure}

For Lyapunov orbits with higher energies we find that the invariant manifolds intersect `faster' than the previous case, the first intersections appear between $W_{1}^{u}$ and $W_{2}^{s}$ and the symmetric intersection $W_{2}^{u}$ and $W_{1}^{s}$. In the Fig. \ref{fourcutsother} we show the first intersections of the invariant manifolds for $C=4.15$.
A similar analysis on how the first cuts of the stable and unstable manifolds depend on the energy level has been done in the case of the R3BP in \cite{GideaM}.
\begin{figure}
\centering
\includegraphics[width=2.0in]{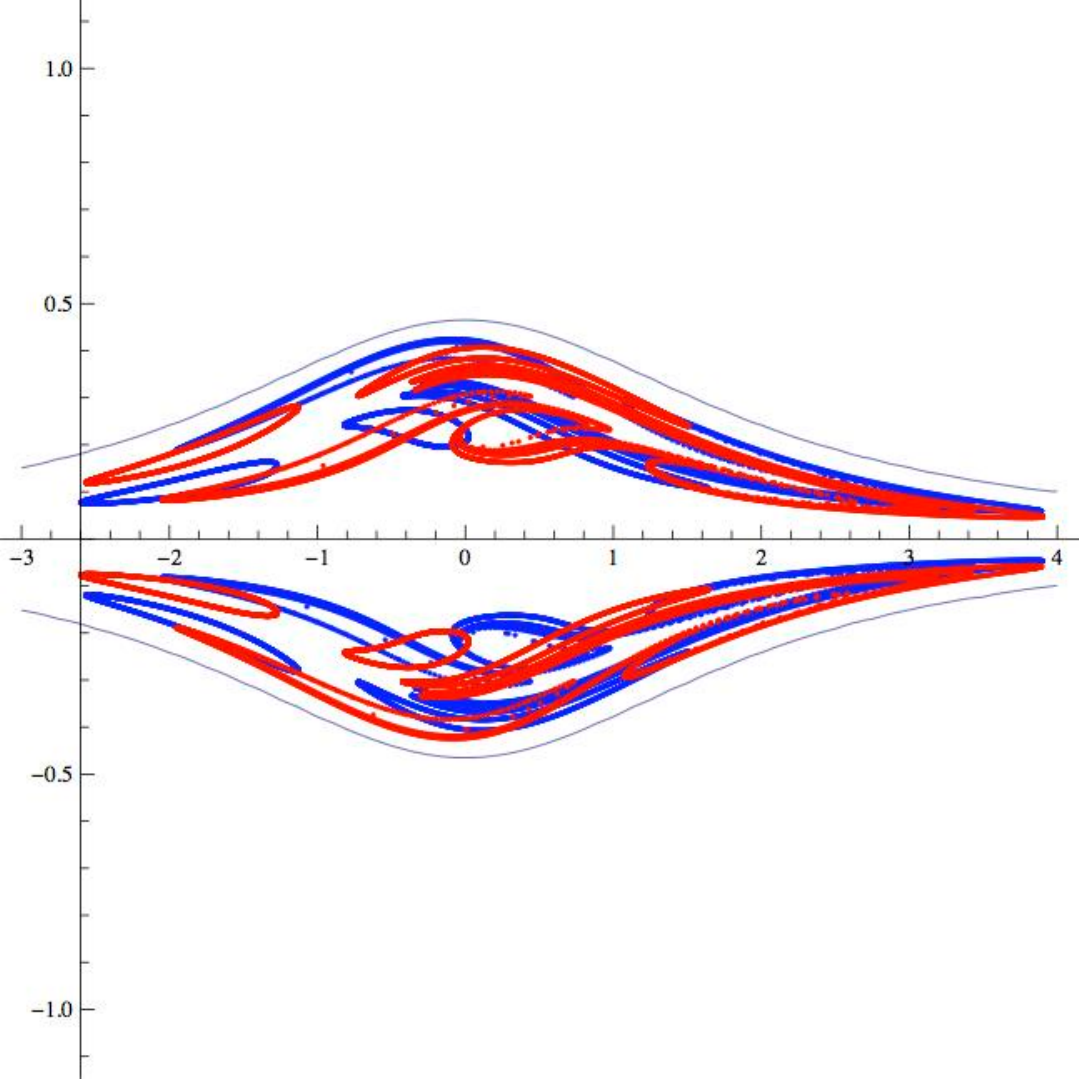}
\caption{The stable and unstable manifolds after ten cuts with the Poincar\'e section in the plane $(\dot{y},y)$.\label{tencuts}}
\end{figure}
\begin{figure}
\centering
\includegraphics[width=2.0in]{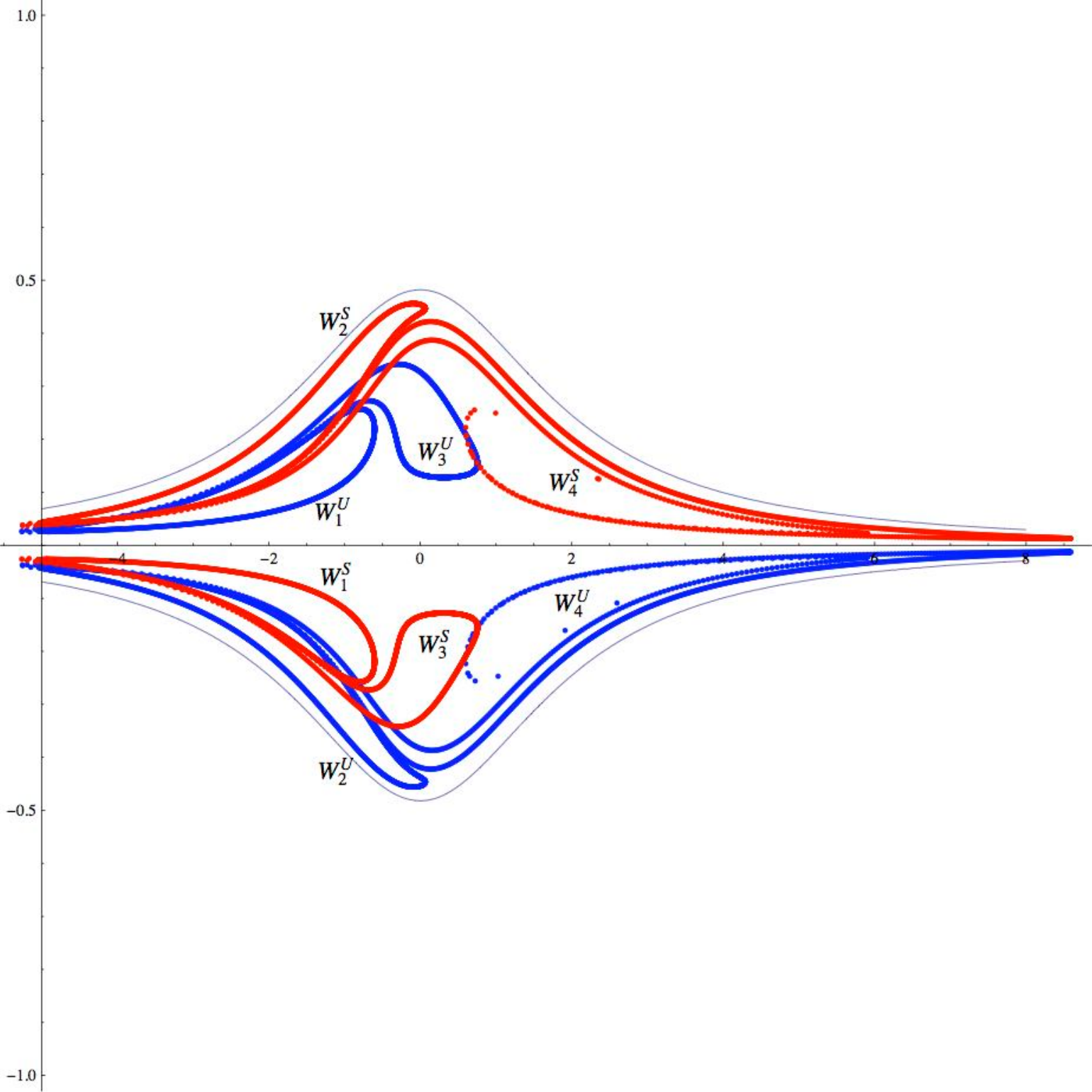}
\caption{The stable and unstable manifolds after four cuts with the Poincar\'e section in the plane $(\dot{y},y)$ for $C=4.15$.\label{fourcutsother}}
\end{figure}

\subsection{The outer region}

The outer region corresponds to the unbounded (blue) region shown in the Fig. \ref{limithillregions}. For this region we consider the branch of the invariant manifolds moving away from the secondary either in forward or backward time. It is worth noting that for the current value of the mass parameter, this branch does not escape to infinity as in the classical Hill problem \cite{Simo}. This behavior is due to the fact in the outer region the gravitational effect of the secondary is small so the dominant part on the dynamics of the infinitesimal mass is the quadratic part of the R3BP from \eqref{quadraticpart}. The invariant manifolds are diffeomorphic to cylinders that turn around close to the zero velocity curve and the behavior of each trajectory on the invariant manifolds is similar to the motion around the equilibrium point $L_{4}$ of the R3BP; see Fig. \ref{manifoldsouterregion}. In order to have a better view of the behavior of the invariant manifolds on the outer region, we  perform the computations in the original (non rotated) coordinates centered at the equilibrium point,  and choose the surface section $\Sigma':=\{(x,y,\dot{x},\dot{y})\in\mathbb{R}^{4}\,\vert\, x=-x_{L_{1}}\}$ which is parallel to the previously considered section $\Sigma$.

\begin{figure}
\centering
\includegraphics[width=2.0in]{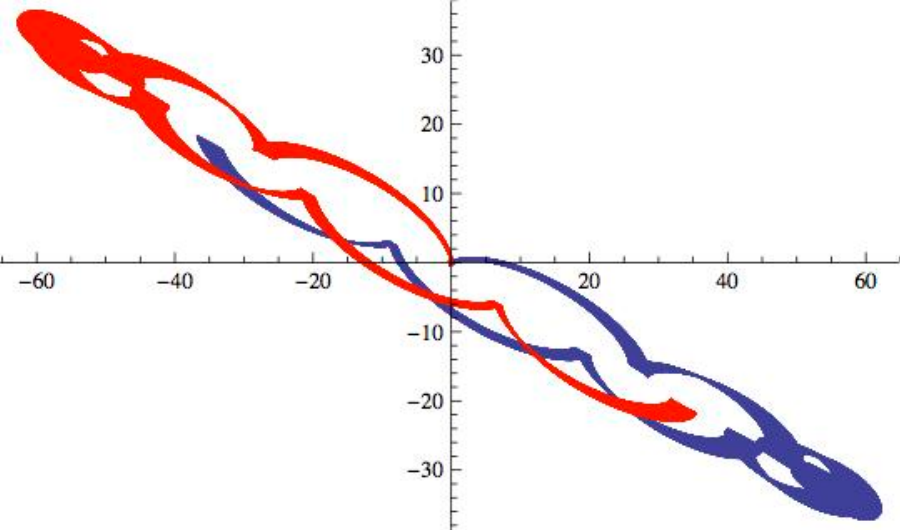}
\caption{Projection of the stable and unstable manifolds on the plane $(x,y)$ in the outer region.\label{manifoldsouterregion}}
\end{figure}
\begin{figure}
\centering
\includegraphics[width=2.5in]{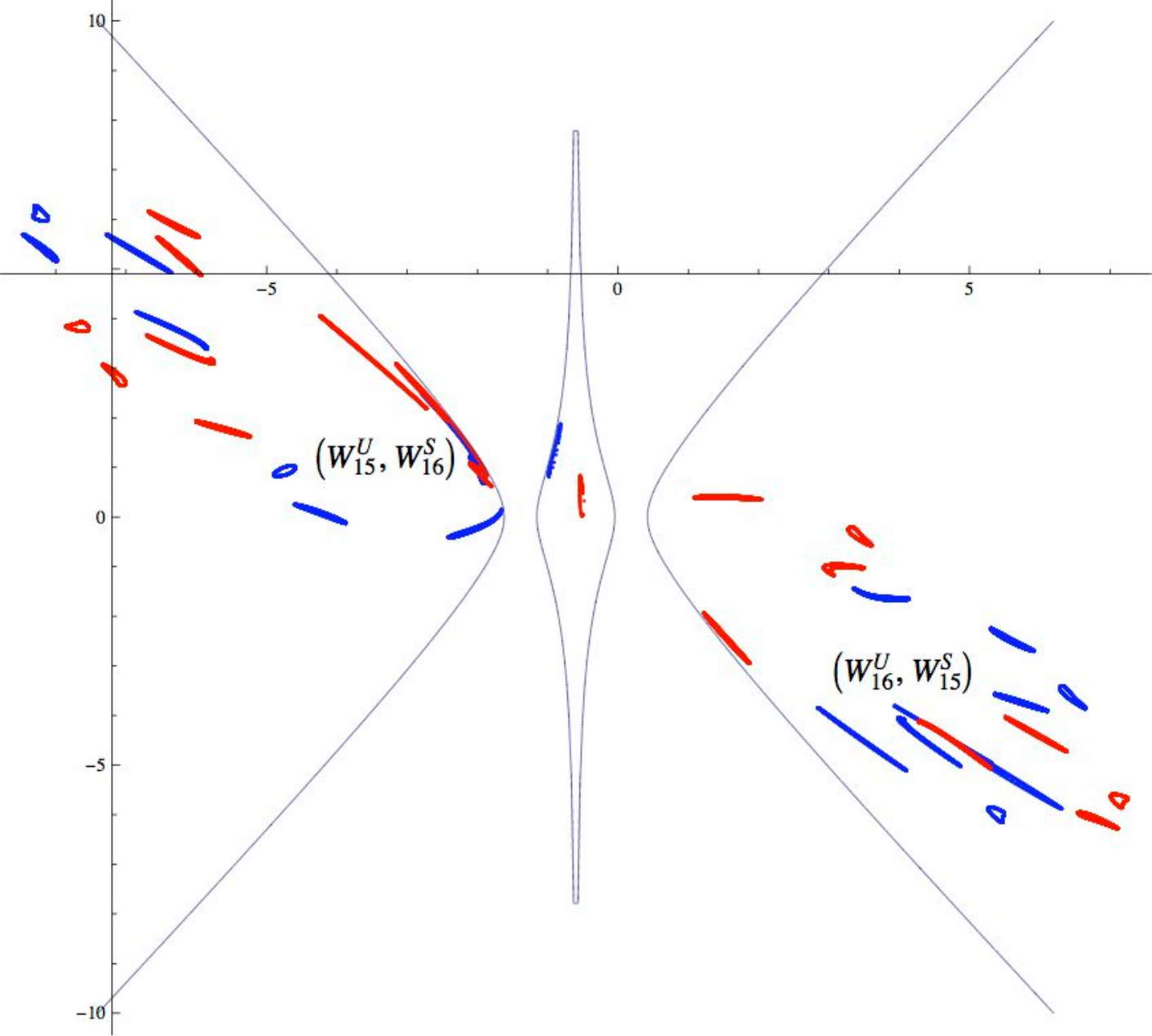}
\caption{Projection of the stable and unstable manifolds on the plane $(y,\dot{y})$ in the outer region.\label{outercuts}}
\end{figure}
\begin{figure}
  \centering
  \includegraphics[width=1.8in]{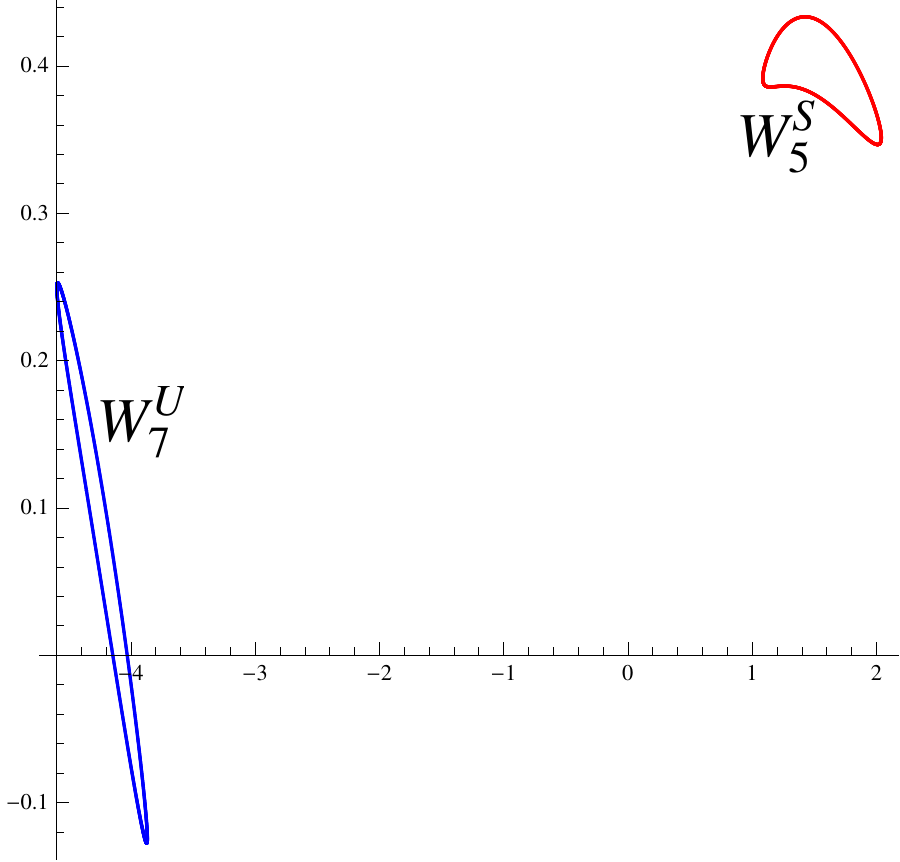}\quad
  \includegraphics[width=1.8in]{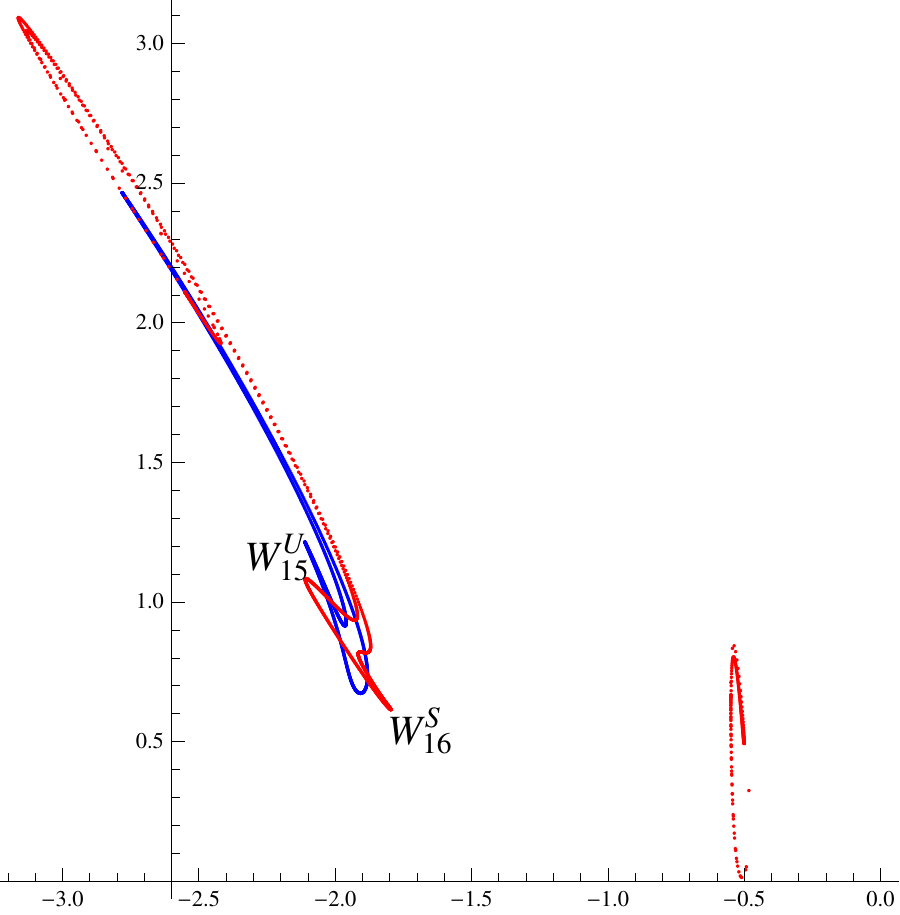}\\
 \caption{Left. Magnification of the cuts $W_{7}^{u}$ and $W_{5}^{s}$.Right. First transverse intersection between the invariant manifolds in the  outer region in the plane $(y,\dot{y})$.
 \label{outer75}}
\end{figure}

In Fig.~\ref{outercuts} we show the projection on the plane $(y,\dot{y})$ of the invariant manifolds after 16 cuts with the surface section $\Sigma'$, the first 14 cuts of the invariant manifolds are curves diffeomorphic to `small' circles, although in the Fig. \ref{outercuts} some of them are indistinguishable. In the Fig.~\ref{outer75} we show a magnification of the cuts $W_{7}^{u}$ and $W_{5}^{s}$.

The first transverse intersection between the invariant manifolds occurs between $W_{15}^{u}$ and $W_{16}^{s}$ (see Fig.~\ref{outer75}) and because of the symmetry respect to the origin of the equations in the original coordinates, we have the corresponding symmetric intersection between $W_{16}^{u}$ and $W_{15}^{s}$.

\section{Summary of results, conclusions and future work}
In this paper a Hill approximation of the restricted four-body problem has been investigated.
In the following we summarize our results and we discuss briefly some applications and future work.

$\bullet$ We have derived our model from the R4BP, where one massless particle moves under the gravitational influence of a nearby small mass (tertiary) and of two distant large masses (primary and secondary) forming an equilateral triangle with the tertiary. We have applied a symplectic scaling   to determine the limit problem when  the mass of the tertiary tends to zero and the primary and the secondary are sent to infinity. In  Theorem \ref{main theorem} we have showed that the limit exists and produces a Hamiltonian that defines our Hill approximation of the  R4BP.

$\bullet$ Our model extends  the classical lunar Hill problem  in the following sense:  when the mass of the secondary  is set to zero, we recover the equations of  classical lunar Hill problem.

$\bullet$ Our model is a good  approximation for the dynamics of the massless particle in a neighborhood of the tertiary. We have proved analytically  the existence of 4 equilibrium points near to the tertiary, as in  the case of the R4BP when the mass of the tertiary is sufficiently small. Also the  Hill  regions for our model are qualitatively the same as  those for  the R4BP.

$\bullet$ We have performed a rigorous study of the linear stability of the equilibrium points in the planar case. We have proved that the collinear equilibrium points are always  unstable, of saddle-center type, for each value of the mass ratio of the secondary vs. the primary. We have also studied the stability of the two `new' equilibrium points (which do not appear in the classical lunar Hill problem) whose  stability  depends on the value of the mass ratio: for values $\mu<\mu_{0}$ these points are center-center type, and for values $\mu>\mu_{0}$ these points are complex-saddle type,  where the threshold value $\mu_0$ has been determined explicitly. For applications to our solar system we note that the mass parameter is less than $\mu_{0}$ and consequently these equilibrium points are linearly stable.

$\bullet$ We have performed numerical experiments  to get an insight into the global behavior of this system. We have applied the Levi-Civita regularization to the equations of motion  and  computed the first return map to a suitable Poincar\'e section  for several values of the mass ratio. These numerical explorations suggest that the current system is non-integrable and exhibits the KAM phenomena. We have also performed a numerical exploration of the invariant manifolds of the Lyapunov orbits for the Jupiter mass parameter $\mu=0.00095$,  showing the formation of transverse intersections between the stable and unstable manifolds of the Lyapunov orbits -- and hence, of homoclinic orbits -- in both the  inner region and in the outer region.  It is worth to note that the branches of the invariant manifolds in the  outer region do not escape to infinity as in the classical lunar Hill problem, but they stay close the zero velocity curves.

$\bullet$ Our model can be used as a first approximation of the dynamics of a spacecraft or small satellite near an asteroid part of a  Sun-Planet-Trojan system. In these cases the value of $\mu$ is less than the  relative mass of Jupiter $\mu=0.00095$. However, our model can be applied to more general systems when the mass parameter can be much bigger, like binary stars systems; see \cite{Sch}. The dynamics of the Trojan asteroids in a neighborhood of the Lagrange points  $L_{4}$ and $L_{5}$ is much more complex, which suggest to include other relevant effects in our model, such as the libration  of the tertiary, inclinations in the spatial problem, perturbations due to oblateness or other bodies, etc. For this, a better insight of the dynamics in our model is required, such as an exploration of the periodic orbits of the system, the non-linear effects on the stability of the equilibrium points, an analytical investigation of the non-integrability of the system, Arnold diffusion, etc. The authors hope to study these problems in future works.

\textbf{Acknowledgements} The research of J.B. has been supported by the CONACyT grant \textit{Estancias posdoctorales en el extranjero}.  The author also wishes to thank Professor Marian Gidea, Mr. Alberto Garcia and the staff of Yeshiva University for their hospitality during the development of this work.

\end{document}